\documentclass[11pt]{article} 
\usepackage{amsfonts,amsmath,latexsym,amssymb,mathrsfs,amsthm,comment}
\usepackage[colorlinks=true, linkcolor=black, citecolor=black]{hyperref}

\let\OLDthebibliography\thebibliography
\renewcommand\thebibliography[1]{
  \OLDthebibliography{#1}
  \setlength{\parskip}{1pt}
  \setlength{\itemsep}{0pt plus 0.0ex}
}

\numberwithin{equation}{section}

\newtheorem{theorem}{Theorem}[section]
\newtheorem{lemma}[theorem]{Lemma}

\newtheorem{definition}[theorem]{Definition}
\newtheorem{proposition}[theorem]{Proposition}
\newtheorem{remark}[theorem]{Remark}
\newtheorem{example}[theorem]{Example}

\usepackage{lscape}
\usepackage{caption}
\usepackage{multirow}
\usepackage{color}

\newcommand{\hs}[1]{{\color{red}#1}}
\allowdisplaybreaks
\begin{document}

\title{Stein's method for asymmetric Laplace approximation
}
\author{Fraser Daly\footnote{Department of Actuarial Mathematics and Statistics, and the Maxwell Institute for Mathematical Sciences, Heriot--Watt University, Edinburgh EH14 4AS. Email: \texttt{F.Daly@hw.ac.uk}},\: Robert E. Gaunt\footnote{Department of Mathematics, The University of Manchester, Manchester M13 9PL, UK. Email: \texttt{robert.gaunt@manchester.ac.uk}; \texttt{heather.sutcliffe@manchester.ac.uk}}\:\, and Heather L. Sutcliff$\mathrm{e}^{\dagger}$}

\date{} 
\maketitle

\vspace{-5mm}

\begin{abstract}
Motivated by its appearance as a limiting distribution for random and non-random sums of independent random variables, in this paper we develop Stein's method for approximation by the asymmetric Laplace distribution. Our results generalise and offer technical refinements on existing results concerning Stein's method for (symmetric) Laplace approximation. We provide general bounds for asymmetric Laplace approximation in the Kolmogorov and Wasserstein distances, and a smooth Wasserstein distance, that involve a distributional transformation that can be viewed as an asymmetric Laplace analogue of the zero bias transformation. As an application, we derive explicit Kolmogorov, Wasserstein and smooth Wasserstein distance bounds for the asymmetric Laplace approximation of geometric random sums, and complement these results by providing explicit bounds for the asymmetric Laplace approximation of a deterministic sum of random variables with a random normalisation sequence. 
\end{abstract}

\noindent{{\bf{Keywords:}}} Stein's method; asymmetric Laplace distribution; asymmetric equilibrium transformation; rate of convergence; 
geometric random sum

\noindent{{{\bf{AMS 2020 Subject Classification:}}} Primary 60F05; 62E17}

\section{Introduction}\label{sec1}



In 1972, Charles Stein introduced a powerful approach to bounding the distance between two probability distributions. His seminal paper \cite{stein}  concerned normal approximation, 
and his method rests on the following characterisation of the standard normal distribution. Define an operator $A$ by $Af(x)=f'(x)-xf(x)$. Then, a real-valued random variable $W$ follows the standard normal distribution if and only if $\mathbb{E}[Af(W)]=0$ for all absolutely continuous functions $f:\mathbb{R}\rightarrow\mathbb{R}$ such that $\mathbb{E}|f'(N)|<\infty$ for $N\sim \mathrm{N}(0,1)$. This \emph{Stein characterisation} of the normal distribution motivates the introduction of the \emph{Stein equation}
$
Af(x)=h(x)-\mathbb{E}[h(N)], 
$
for some test function $h:\mathbb{R}\rightarrow\mathbb{R}$. The function $f_h(x)=-\mathrm{e}^{x^2/2}\int_{-\infty}^x (h(t)-\mathbb{E}[h(t)])\mathrm{e}^{-t^2/2}\,\mathrm{d}t$ solves the Stein equation, and its regularity properties are now very well understood; see, for example, \cite{chen,g25}. One can then bound the distance between the distribution of a given random variable $W$ and the standard normal distribution with
respect to an integral probability metric via the transfer principle
\begin{equation}\label{transf}
d_{\mathcal{H}}(W,N):= \sup_{h\in\mathcal{H}}|\mathbb{E}[h(W)]-\mathbb{E}[h(N)]|= \sup_{h\in\mathcal{H}}|\mathbb{E}[Af_h(W)]|,  
\end{equation}
where $\mathcal{H}$ is some class of test functions; for example, taking $\mathcal{H}_{\mathrm{K}}=\{\mathbb{I}_{\cdot\leq z}\,|\,z\in\mathbb{R}\}$, $\mathcal{H}_{\mathrm{W}}=\{h:\mathbb{R}\rightarrow\mathbb{R}\,|\,\text{$h$ is Lipschitz, $\|h'\|\leq1$}\}$ and $\mathcal{H}_{\mathrm{bW}}=\mathcal{H}_{\mathrm{W}}\cap\{h:\mathbb{R}\rightarrow\mathbb{R}\,|\,\text{$\|h\|\leq1$}\}$ induces the Kolmogorov, Wasserstein and bounded Wasserstein distances, which we will denote by $d_{\mathrm{K}}$, $d_\mathrm{W}$ and $d_{\mathrm{bW}}$, respectively. Here and throughout this paper, $\|g\|$ will denote the essential supremum norm on $\mathbb{R}$ of a function $g:\mathbb{R}\rightarrow\mathbb{R}$, so that if $g$ is a Lipschitz function then $\|g'\|$ is its minimum Lipschitz constant. 
In this paper, we will also employ the  Wasserstein-type distances $d_2$ and $d_{1,2}$ (sometimes referred to as smooth Wasserstein distances), which are induced by the function classes $\mathcal{H}_{2}=\{h:\mathbb{R}\rightarrow\mathbb{R}\,|\,h:\mathbb{R}\rightarrow\mathbb{R}\,|\,\text{$h'$ is Lipschitz, $\|h''\|\leq1$}\}$ and $\mathcal{H}_{1,2}=\mathcal{H}_{\mathrm{W}}\cap \mathcal{H}_{2}$, respectively.

The success of Stein's method rests on the fact that the right-hand side of (\ref{transf}) involves a single random variable $W$ and it is typically easier to bound this expectation than to directly bound the integral probability metric $d_{\mathcal{H}}(W,N)$. For introductions to Stein's method and its numerous applications across the mathematical sciences we refer to the monographs \cite{chen,np12} and the survey \cite{ross}. A further advantage of Stein's method is that it applies to many target distributions beyond the normal; indeed, over the years Stein's method has been adapted to many distributions from the classical Poisson \cite{c75}, exponential \cite{chatterjee,pekoz1}, chi-square \cite{gaunt chi square} and beta \cite{d15,gr13} distributions to more exotic distributions like the variance-gamma \cite{gaunt vg} and stable distributions \cite{ah19,xu19}.

In this paper, we contribute to the extensive literature concerning the extension of Stein's method to new distributions by developing Stein's method for asymmetric Laplace approximation. The asymmetric Laplace distribution with parameters $a\in\mathbb{R}$, $b>0$ and $\mu\in\mathbb{R}$, which we denote by $\mathrm{AL}(\mu,a,b)$, is often defined via its characteristic function
\begin{equation}\label{cf1}
	\phi_W(t)=\mathbb{E}[\mathrm{e}^{\mathrm{i}tW}]=\frac{\mathrm{e}^{\mathrm{i}\mu t}}{1+b^2t^2/2-\mathrm{i}at}, \quad t\in\mathbb{R}.
\end{equation}
In this parametrisation, $\mu$ is a location parameter, $b$ is a scale parameter and $a$ is an asymmetry parameter; the distribution is symmetric about the location parameter $\mu$ when $a=0$. The probability density function of $W\sim \mathrm{AL}(\mu,a,b)$ is given by
\begin{equation}\label{alpdf}
	f_W(x)=\frac{1}{\sqrt{2b^2+a^2}}\exp\bigg(\frac{a}{b^2}(x-\mu)-\frac{1}{b}\sqrt{2+\frac{a^2}{b^2}}|x-\mu|\bigg), \quad x\in\mathbb{R},
\end{equation}
whilst the mean and variance are given by
\begin{equation*}
\mathbb{E}[W]=a+\mu, \quad \mathrm{Var}(W)=b^2+(a+\mu)^2.    
\end{equation*}
The asymmetric Laplace distribution can be represented as a normal variance-mean mixture. Let $X$ and $N$ be independent random variables following the exponential distribution with unit rate parameter and the standard normal distribution, respectively. Then
\begin{equation}
\mu+aX+b\sqrt{X}N\sim \mathrm{AL}(\mu,a,b).   \label{nvm} 
\end{equation}
These and further properties, along with accounts of the numerous applications of the asymmetric Laplace distribution, are given in the monograph \cite{kkp01}.

One of the most significant applications of the asymmetric Laplace distribution is to financial modelling; see, for example, \cite{kp01,kr94,t12} and \cite[Chapter 8]{kkp01}. In financial applications, the asymmetric Laplace distribution often arises in the following manner \cite[Section 8.3]{kkp01}. Following \cite{mr93}, one may consider a stock price change as a sum of a large number of small changes, in which the number of small changes is taken up to a random time $N_p$. Informally stated, 
\begin{equation*}
\text{stock price change}=\sum_{i=1}^{N_p}(\text{small changes}) .  
\end{equation*}
The random time $N_p$ represents the time at which the probabilistic structure describing the stock price breaks down (for example, due to a change in the fundamentals of the underlying asset). Taking $N_p$ to be geometrically distributed (so that $p$ denotes the small probability of a dramatic change in a market) motivates modelling stock price changes via an asymmetric Laplace distribution (see, for example, \cite[Proposition 3.4.4]{kkp01}), which has been found empirically to be a good fit to real financial data; see, for example, \cite{kr94}. 

Motivated by the appearance of the asymmetric Laplace distribution in financal modelling, \cite{toda} derived a Lindeberg-type condition under which a suitably normalised geometric random sum of independent random variables converges in distribution to the asymmetric Laplace law. Shortly afterwards, \cite{pike} 
developed Stein's method for (symmetric) Laplace approximation (the $a=0$ case) and applied this theory to derive an explicit bounded Wasserstein distance bound to quantify the Laplace approximation of a geometric random sum of independent zero mean random variables. Technical refinements to the treatment of \cite{pike} were subsequently made independently by \cite{g21,s21} that allowed for bounds to be stated with respect to stronger probability metrics. Further works have also given bounds to quantify this distributional approximation \cite{bu24,d15,g20,ks12,s18}. However, to the best of our knowledge, bounds that quantify the more general asymmetric Laplace approximation of a geometric random sum of independent random variables with possibly non-zero means have yet to be given (note that in the context of financial modelling the `small random changes' in a stock price will typically have a small non-zero mean). 

Our development of Stein's method for asymmetric Laplace approximation is a natural extension of the papers \cite{g21,pike,s21}.
The asymmetric Laplace distribution is one of the most important classical distributions for which Stein's method is yet to be developed, and arises as a limiting distribution for both geometric random sums and sums of independent random variables with random standardisation (see Sections \ref{sec4.1} and \ref{sec4.2}, respectively). Our work fills this gap in the Stein's method literature. 
It should be remarked that several works have developed Stein's method for classes of distributions that include the asymmetric Laplace distribution as special cases, such as the variance-gamma \cite{gaunt vg} and infinitely divisible distributions \cite{ah19}. However, as is to be expected, the specialist treatment provided in this paper can lead to stronger results in the asymmetric Laplace setting; see Remark \ref{remvg} for a discussion. 

In Section \ref{sec2}, we obtain a Stein characterisation of the asymmetric Laplace distribution (Lemma \ref{lem2.1}), from which we are able to write down a Stein equation for the asymmetric Laplace distribution. Regularity estimates for the solution of the Stein equation are given in Lemma \ref{lem2.2}. 

In Section \ref{sec3}, we introduce a distributional transformation that is well suited to asymmetric Laplace approximation via Stein's method. This distributional transformation, which we refer to as the asymmetric equilibrium transformation (Definition \ref{def1}), is a natural generalisation of the centered equilibrium transformation introduced by \cite{pike} for Laplace approximation, and can also be viewed as an asymmetric Laplace analogue of the zero bias transformation that was introduced by \cite{gr97} for normal approximation. We establish some basic properties of the asymmetric equilibrium transformation in Proposition \ref{prop3.3} and provide general plug-in bounds for asymmetric Laplace approximation, expressed in terms of this distributional transformation, in Theorem \ref{thm3.5}.

Section \ref{sec4} is devoted to applications of the preceding work, beginning with an
application to perturbations of the asymmetric Laplace distribution in Section \ref{sec4.0}. In Section \ref{sec4.1}, we apply our general bounds of Theorem \ref{thm3.5} to derive explicit bounds for the asymmetric Laplace approximation of a geometric random sum of independent random variables.
In Section \ref{sec4.2}, we complement these bounds by deriving explicit bounds for the asymmetric Laplace approximation of a deterministic sum of random variables with a random normalisation sequence. 





\begin{remark}
During the course of this research, the second author was informed by Nathan Ross that his MSc student, Xin Huang, had independently written a dissertation on Stein's method and the asymmetric Laplace distribution \cite{huang}, developing the basic ingredients required to set up Stein's method for asymmetric Laplace approximation with an application to geometric random sums. In particular, the dissertation proves Lemma \ref{lem2.1} with different conditions, provides some bounds for the solution of the Stein equation (with larger constants and stronger assumptions on the test function $h$ than those given in Lemma \ref{lem2.2}), gives a proof of the construction of the asymmetric equilibrium transformation as in part (ii) of Proposition \ref{prop3.3} under an additional assumption, and gives an analogue of the Wasserstein distance bound of Theorem \ref{thm4.1} in the weaker bounded Wasserstein distance under stronger assumptions. Since the dissertation \cite{huang} is not publicly available and all our corresponding results are stronger in some sense, we provide full independent proofs of all these results.
\end{remark}

\section{A Stein equation for the asymmetric Laplace distribution}\label{sec2}

We begin by establishing a Stein characterisation of the asymmetric Laplace distribution.

\begin{lemma}\label{lem2.1}Let $W$ be a real-valued random variable. Then $W$ follows the $\mathrm{AL}(\mu,a,b)$ distribution if and only if
	\begin{equation}\label{alapchar} \mathbb{E}\bigg[\frac{b^2}{2}f''(W)+af'(W)-\big(f(W)-f(\mu)\big)\bigg]=0
	\end{equation}
	for all functions $f:\mathbb{R}\rightarrow\mathbb{R}$ such that $f$ and $f'$ are locally absolutely continuous with $\mathbb{E}|f'(Z)|<\infty$ and $\mathbb{E}|f''(Z)|<\infty$, for $Z\sim\mathrm{AL}(\mu,a,b)$.
\end{lemma}

Our Stein characterisation leads to a Stein equation for the asymmetric Laplace distribution, which we state as equation (\ref{ivp}) in the following lemma. The lemma provides a solution to the Stein equation together with regularity estimates for the solution. In this lemma and henceforth, we will without loss of generality set $\mu=0$.

\begin{lemma}\label{lem2.2}Let $h:\mathbb{R}\rightarrow\mathbb{R}$ be a measurable function such that $\mathbb{E}|h(Z)|<\infty$, for $Z\sim\mathrm{AL}(0,a,b)$.  Then a solution to the initial value problem 
	\begin{equation}\label{ivp}\frac{b^2}{2}f''(x)+ af'(x)-f(x)=h(x)-\mathbb{E}[h(Z)], \quad f(0)=0
	\end{equation}
	is given by
	\begin{align}f_h(x)=-\frac{1}{\sqrt{2b^2+a^2}}\bigg(&\mathrm{e}^{\frac{1}{b}(\sqrt{2+a^2/b^2}-a/b)x}\int_x^\infty \mathrm{e}^{-\frac{1}{b}(\sqrt{2+a^2/b^2}-a/b)t}\tilde{h}(t)\,\mathrm{d}t\nonumber \\
		\label{sol}&+\mathrm{e}^{-\frac{1}{b}(\sqrt{2+a^2/b^2}+a/b)x}\int_{-\infty}^x \mathrm{e}^{\frac{1}{b}(\sqrt{2+a^2/b^2}+a/b)t}\tilde{h}(t)\,\mathrm{d}t\bigg),
	\end{align}
where $\tilde{h}(x)=h(x)-\mathbb{E}[h(Z)]$.    
	If $h$ is bounded, then $f_h$ is the unique bounded solution to (\ref{ivp}),
    and the following bounds hold:
	\begin{equation}\|f_h\|\leq\|\tilde{h}\|, \quad \|f_h'\|\leq\frac{2}{\sqrt{2b^2+a^2}}\|\tilde{h}\|, \quad \|f_h''\|\leq\frac{4}{b^2}\|\tilde{h}\|. \label{kolstein}
	\end{equation}
Let $k\geq0$ and now suppose that $h^{(k)}$ is Lipschitz. Then, 
	\begin{equation}\label{2.6}\|f_h^{(k+1)}\|\leq\|h^{(k+1)}\|, \quad \|f_h^{(k+2)}\|\leq\frac{2}{\sqrt{2b^2+a^2}}\|h^{(k+1)}\|, \quad \|f_h^{(k+3)}\|\leq\frac{4}{b^2}\|h^{(k+1)}\|,
	\end{equation}
	where $h^{(0)}\equiv h$.
\end{lemma}


\begin{remark}\label{remvg} Taking $f(x)=(x-\mu)g(x)$ in the initial value problem (\ref{ivp}) leads to the following alternative Stein equation for the $\mathrm{AL}(\mu,a,b)$ distribution:
\begin{equation}\label{alst2}
 \frac{b^2}{2}(x-\mu)g''(x)+(b^2+a(x-\mu))g'(x)+(a-(x-\mu))g(x)=h(x)-\mathbb{E}[h(Z)].  
\end{equation}
The Stein equation (\ref{alst2}) is a special case of the variance-gamma Stein equation of \cite{gaunt vg}. The variance-gamma Stein equation of \cite{gaunt vg} has been used to obtain quantitative six moment theorems for the variance-gamma approximation of double Wiener–It\^o integrals \cite{aet22,et15,g22}, which as a special case provide quantitative asymmetric Laplace approximations of double Wiener–It\^o integrals. However, the Stein equation (\ref{alst2}) can be difficult to work with. Indeed, the problem of obtaining suitable regularity estimates for the solution of the variance-gamma Stein equation was rather involved and spanned four papers \cite{dgv15,gaunt vg,g20,g22}, whilst in Lemma \ref{lem2.2} we are able to provide an efficient proof of accurate bounds (with good dependence on the parameters $a$ and $b$ and the test function $h$) for the solution of the asymmetric Laplace Stein equation (\ref{ivp}). In principle, it would be possible to develop the theory along the lines of the forthcoming Section \ref{sec3} and prove the applications of Theorems \ref{thm:perturb} and \ref{thm4.1} using the Stein equation (\ref{alst2}), but the treatment would be more involved and lead to worse constants, and the approach may lead to bounds that are given in weaker probability metrics. 

We also remark that \cite{bu24} recently obtained an alternative Stein characterisation of the (symmetric) Laplace distribution as a special case of a Stein characterisation of \cite[Theorem 3.1]{ah19} for infinitely divisible distributions. The Stein characterisation of \cite{bu24} leads to an integral Stein equation for the Laplace distribution, and they applied their Stein equation to the geometric random sum example we treat in Section \ref{sec4.1} in the case of symmetric Laplace approximation. However, by working with their integral Stein operator, they can only give an (optimal order) bound with respect to a strictly weaker probability metric than those that we work with in Theorem \ref{thm4.1}. Therefore working with a generalisation of the Stein characterisation of \cite{bu24} and following their approach would lead to weaker results than we are able to achieve through our approach that rests on the Stein characterisation given in Lemma \ref{lem2.1}.
\end{remark}


\noindent{\emph{Proof of Lemma \ref{lem2.1}}.}  \emph{Necessity.} We prove the lemma for the case $\mu=0$; the general case $\mu\in\mathbb{R}$ follows from a simple translation. It will be notationally convenient to let $\alpha=b^{-1}\sqrt{2+a^2/b^2}$ and $\beta=a/b^2$. With this change of variables, we are required to prove that $\mathbb{E}[f''(W)+2\beta f'(W)-(\alpha^2-\beta^2)(f(W)-f(0))]=0$ for $W\sim\mathrm{AL}(0,a,b)$. We note that in this parameterisation the probability density function of the $\mathrm{AL}(0,a,b)$ distribution is given by $(2\alpha)^{-1}(\alpha^2-\beta^2)\mathrm{e}^{\beta x-\alpha|x|}$, for $x\in\mathbb{R}$. By applying Fubini's theorem, we have that 
\begin{align*}
	\int_{0}^{\infty}f'(x)\mathrm{e}^{-(\alpha-\beta)x}\,\mathrm{d}x&=\int_{0}^{\infty}f'(x)\bigg(\int_{x}^{\infty}(\alpha-\beta)\mathrm{e}^{-(\alpha-\beta)y}\,\mathrm{d}y\bigg)\,\mathrm{d}x\\
	&=(\alpha-\beta)\int_{0}^{\infty}\int_{0}^{y}f'(x)\mathrm{e}^{-(\alpha-\beta)y}\,\mathrm{d}x\,\mathrm{d}y\\
	&=(\alpha-\beta)\int_{0}^{\infty}f(y)\mathrm{e}^{-(\alpha-\beta)y}\,\mathrm{d}y-f(0), 
\end{align*}
from which we then readily obtain that
\begin{align*}
	\int_{0}^{\infty}f''(x)\mathrm{e}^{-(\alpha-\beta)x}\,\mathrm{d}x=(\alpha-\beta)^2\int_{0}^{\infty}f(y)\mathrm{e}^{-(\alpha-\beta) y}\,\mathrm{d}y-(\alpha-\beta)f(0)-f'(0).
\end{align*}
By similar calculations we obtain the expressions
\begin{align*}
\int_{-\infty}^{0}f'(x)\mathrm{e}^{(\alpha+\beta)x}\,\mathrm{d}x&=-(\alpha+\beta)\int_{-\infty}^{0}f(y)\mathrm{e}^{(\alpha+\beta)y}\,\mathrm{d}y+f(0), \\
\int_{-\infty}^{0}f''(x)\mathrm{e}^{(\alpha+\beta)x}\,\mathrm{d}x&=(\alpha+\beta)^2\int_{-\infty}^{0}f(y)\mathrm{e}^{(\alpha+\beta)y}\,\mathrm{d}y-(\alpha+\beta)f(0)+f'(0).    
\end{align*}
On combining the above expressions and simplifying we obtain that
\begin{align*}
	&\mathbb{E}[f''(W)]+2\beta\mathbb{E}[f'(W)]\\
	&\quad=(\alpha^2-\beta^2)\bigg\{\frac{\alpha^2-\beta^2}{2\alpha}\int_{0}^{\infty}f(y)\mathrm{e}^{-(\alpha-\beta)y}\,\mathrm{d}y+\frac{\alpha^2-\beta^2}{2\alpha}\int_{-\infty}^{0}f(y)\mathrm{e}^{(\alpha+\beta)y}\,\mathrm{d}y-f(0)\bigg\}\\
	&\quad=(\alpha^2-\beta^2)\big(\mathbb{E}[f(W)]-f(0)\big),
\end{align*}
as required.

\vspace{3mm}

\noindent \emph{Sufficiency.} Let $W$ be a real-valued random variable satisfying \eqref{alapchar}, and denote its characteristic function by $\phi_W(t)=\mathbb{E}[\mathrm{e}^{\mathrm{i}tW}]$. Applying the Stein characterisation (\ref{alapchar}) to the real and imaginary parts of the function $f(x)=\mathrm{e}^{\mathrm{i}tx}$ (noting that $f$ and $f'$ are locally absolutely continuous with $\mathbb{E}|f'(Z)|<\infty$ and $\mathbb{E}|f''(Z)|<\infty$, for $Z\sim\mathrm{AL}(0,a,b)$) now reveals that $\phi_W(t)$ satisfies the equation
\begin{equation*}
	\bigg(-\frac{b^2}{2}t^2+a\mathrm{i}t-1\bigg)\phi_W(t)+1=0,
\end{equation*}
which, upon rearranging, has solution
	$\phi_W(t)=(1-a\mathrm{i}t+b^2t^2/2)^{-1}$.
This is the characteristic function of the $\mathrm{AL}(0,a,b)$ distribution (\ref{cf1}), and it thus follows from the uniqueness of characteristic functions that $W\sim\mathrm{AL}(0,a,b)$.\qed

\vspace{3mm}

\noindent{\emph{Proof of Lemma \ref{lem2.2}}.} It will again be notationally convenient to let $\alpha=b^{-1}\sqrt{2+a^2/b^2}$ and $\beta=a/b^2$, so that the initial value problem (\ref{ivp}) becomes
\begin{equation}\label{ivp2}f''(x)+2\beta f'(x)-(\alpha^2-\beta^2)f(x)=\frac{2}{b^2}\tilde{h}(x), \quad f(0)=0.
\end{equation}

Firstly, we prove that there is at most one bounded solution to the initial value problem (\ref{ivp2}).  Suppose $u$ and $v$ are solutions to (\ref{ivp2}).  Then $w=u-v$ satisfies the initial value problem
\begin{equation}\label{ivp3}w''(x)+2\beta w'(x)-(\alpha^2-\beta^2)w(x)=0, \quad w(0)=0.
\end{equation}
The initial value problem (\ref{ivp3}) has family of solutions $w(x)=A(\mathrm{e}^{(\alpha-\beta)x}-\mathrm{e}^{-(\alpha+\beta)x})$, where $A\in\mathbb{R}$ is a constant.  To ensure that $w$ is bounded we must take $A=0$, and so $w=0$, meaning that there is at most one bounded solution to (\ref{ivp2}). 

Now, since (\ref{ivp2}) is an inhomogeneous linear ordinary differential equation, we can use the method of variation of parameters (see \cite{collins} for an account of the method) to write down the general solution:
\begin{equation}\label{ffviv}f(x)=\frac{2}{b^2}\bigg(-w_1(x)\int_c^x\frac{w_2(t)\tilde{h}(t)}{W(t)}\,\mathrm{d}t+w_2(x)\int_d^x\frac{w_1(t)\tilde{h}(t)}{W(t)}\,\mathrm{d}t\bigg),
\end{equation}
where $w_1$ and $w_2$ are linearly independent solutions to the homogeneous equation (that is the differential equation in (\ref{ivp2}) with the RHS set to zero), $c$ and $d$ are arbitrary constants (which must be such that the condition $f(0)=0$ is satisfied), and $W(t)=w_1(t)w_2'(t)-w_2(t)w_1'(t)$ is the Wronskian.  We find that $w_1(x)=\mathrm{e}^{(\alpha-\beta)x}$ and $w_2(x)=\mathrm{e}^{-(\alpha+\beta)x}$ are linearly independent solutions to the homogeneous equation, and $W(t)=-2\alpha\mathrm{e}^{-2\beta t}$. Therefore, on substituting these quantities into (\ref{ffviv}) and taking $c=\infty$ and $d=-\infty$, we obtain the solution
\begin{equation}\label{vargensoln}f_h(x)=-\frac{1}{\alpha b^2}\bigg(\mathrm{e}^{(\alpha-\beta)x}\int_x^\infty\mathrm{e}^{-(\alpha-\beta)t}\tilde{h}(t)\,\mathrm{d}t+\mathrm{e}^{-(\alpha+\beta)x}\int_{-\infty}^x\mathrm{e}^{(\alpha+\beta)t}\tilde{h}(t)\,\mathrm{d}t\bigg),
\end{equation}
which proves that (\ref{sol}) solves the differential equation in (\ref{ivp}).  To see that the initial condition is satisfied, we note that
\begin{align*}f_h(0)=-\frac{1}{\sqrt{2b^2+a^2}}\int_{-\infty}^\infty\mathrm{e}^{\frac{1}{b}(at/b-\sqrt{2+a^2/b^2}|t|)}\tilde{h}(t)\,\mathrm{d}t=0,
\end{align*} 
where the second equality holds because the integrand is proportional to the $\mathrm{AL}(0,a,b)$ density (\ref{alpdf}).

We now derive the regularity estimates for the solution $f_h$. We continue to work in our $(\alpha,\beta)$ parametrisation. We first suppose that $h$ is bounded.  Then, for $x\in\mathbb{R}$,
\begin{align}\bigg|\mathrm{e}^{(\alpha-\beta)x}\int_x^\infty \mathrm{e}^{-(\alpha-\beta)t}\tilde{h}(t)\,\mathrm{d}t\bigg|&\leq \|\tilde{h}\|\mathrm{e}^{(\alpha-\beta)x}\int_x^\infty \mathrm{e}^{-(\alpha-\beta)t}\,\mathrm{d}t=\frac{\|\tilde{h}\|}{\alpha-\beta},\label{bb1} \\
\bigg|\mathrm{e}^{-(\alpha+\beta)x}\int_{-\infty}^x\mathrm{e}^{(\alpha+\beta)t}\tilde{h}(t)\,\mathrm{d}t\bigg|&\leq \|\tilde{h}\|\mathrm{e}^{-(\alpha+\beta)x}\int_{-\infty}^x \mathrm{e}^{(\alpha+\beta)t}\,\mathrm{d}t=\frac{\|\tilde{h}\|}{\alpha+\beta}. \label{bb2}
\end{align}
By applying these bounds to the solution (\ref{vargensoln}) we obtain the bound
\begin{equation*}\label{fequal}\|f_h\|\leq\frac{\|\tilde{h}\|}{\alpha b^2}\bigg(\frac{1}{\alpha-\beta}+\frac{1}{\alpha+\beta}\bigg)=\frac{\|\tilde{h}\|}{b^2}\frac{2}{(\alpha^2-\beta^2)}=\|\tilde{h}\|.
\end{equation*}
The derivative of $f_h$ is given by
\begin{equation}\label{firstd}f_h'(x)=-\frac{1}{\alpha b^2}\bigg((\alpha-\beta)\mathrm{e}^{(\alpha-\beta)x}\int_x^\infty\mathrm{e}^{-(\alpha-\beta)t}\tilde{h}(t)\,\mathrm{d}t-(\alpha+\beta)\mathrm{e}^{-(\alpha+\beta)x}\int_{-\infty}^x\mathrm{e}^{(\alpha+\beta)t}\tilde{h}(t)\,\mathrm{d}t\bigg),
\end{equation}
and by applying inequalities (\ref{bb1}) and (\ref{bb2}) we obtain the bound
\begin{equation*}\label{fdash}\|f_h'\|\leq\frac{2\|\tilde{h}\|}{\alpha b^2}=\frac{2\|\tilde{h}\|}{\sqrt{2b^2+a^2}}.
\end{equation*}
On rearranging (\ref{ivp2}) and applying formulas (\ref{vargensoln}) and (\ref{firstd}) we obtain that, for any $x\in\mathbb{R}$,
\begin{align*}|f_h''(x)|&=\bigg|\frac{2}{b^2}\tilde{h}(x)+(\alpha^2-\beta^2)f_h(x)-2\beta f_h'(x)\bigg| \\
	&=\bigg|\frac{2}{b^2}\tilde{h}(x)-\frac{1}{\alpha b^2}\bigg((\alpha-\beta)^2\mathrm{e}^{(\alpha-\beta)x}\int_x^\infty\mathrm{e}^{-(\alpha-\beta)t}\tilde{h}(t)\,\mathrm{d}t\\
	&\quad+(\alpha+\beta)^2\mathrm{e}^{-(\alpha+\beta)x}\int_{-\infty}^x\mathrm{e}^{(\alpha+\beta)t}\tilde{h}(t)\,\mathrm{d}t\bigg)\bigg| \\
	&\leq\frac{2}{b^2}\|\tilde{h}\|+\frac{1}{\alpha b^2}\big((\alpha-\beta)\|\tilde{h}\|+(\alpha+\beta)\|\tilde{h}\|\big)=\frac{4}{b^2}\|\tilde{h}\|.
\end{align*}

Suppose now that $h$ is Lipschitz.
Manipulating the integrals in (\ref{firstd}) by an application of integration by parts yields the expression
\begin{align*}f_h'(x)&=-\frac{1}{\alpha b^2}\bigg\{(\alpha-\beta)\mathrm{e}^{(\alpha-\beta)x}\bigg(\frac{1}{\alpha-\beta}\mathrm{e}^{-(\alpha-\beta)x}\tilde{h}(x)+\frac{1}{\alpha-\beta}\int_x^\infty\mathrm{e}^{-(\alpha-\beta)t}h'(t)\,\mathrm{d}t\bigg)\\
	&\quad-(\alpha+\beta)\mathrm{e}^{-(\alpha+\beta)x}\bigg(\frac{1}{\alpha+\beta}\mathrm{e}^{(\alpha+\beta)x}\tilde{h}(x)-\frac{1}{\alpha+\beta}\int_{-\infty}^x\mathrm{e}^{(\alpha+\beta)t}h'(t)\,\mathrm{d}t\bigg)\bigg\} \\
	&=-\frac{1}{\alpha b^2}\bigg(\mathrm{e}^{(\alpha-\beta)x}\int_x^\infty\mathrm{e}^{-(\alpha-\beta)t}h'(t)\,\mathrm{d}t+\mathrm{e}^{-(\alpha+\beta)x}\int_{-\infty}^x\mathrm{e}^{(\alpha+\beta)t}h'(t)\,\mathrm{d}t\bigg).
\end{align*}
This representation of the derivative $f_h'(x)$ is the same as the representation (\ref{vargensoln}) of the solution $f_h(x)$, with the sole difference being that $h'(t)$ replaces the role of $\tilde{h}(t)$ in the integrands. We thus immediately arrive at the bounds given in (\ref{2.6}) in the case $k=0$. Iterating this procedure then gives the desired bounds for $\|f_h^{(k+1)}\|$, $\|f_h^{(k+2)}\|$ and $\|f_h^{(k+3)}\|$, for $k\geq1$.
\qed

\section{The asymmetric equilibrium transformation}\label{sec3}

In this section, we define a distributional transformation that will be one of the key tools in the work that follows. Such distributional transformations have played an important role in many different settings in which Stein's method has been employed; see the survey \cite{ross} for a discussion of several of these. The transformation we define here, called the \emph{asymmetric equilibrium transformation}, generalises an earlier definition of \cite{pike} used in their Laplace approximation results. We begin in Section \ref{sec3.1} by stating the definition and giving a number of fundamental properties of the asymmetric equilibrium transformation. This transformation has the asymmetric Laplace distribution as its unique fixed point, and as such many of our approximation results rely on couplings between a given random variable and its asymmetric equilibrium version. In Section \ref{sec3.2}, we explore several settings in which we may explicitly couple these random variables, including the setting of random sums which will be important in the applications that follow. Finally, in Section \ref{sec3.3}, we give general bounds that make use of such coupling constructions in order to prepare for the applications of Section \ref{sec4}.  

\subsection{Definition and useful properties}\label{sec3.1}

\begin{definition}\label{def1}
	Let $W$ be a random variable with mean $a\in\mathbb{R}$ and variance $\sigma^2\in(0,\infty)$ and suppose that $\sigma^2>a^2$. Then we will say that the random variable $W^A$ has the \emph{asymmetric equilibrium distribution} with respect to $W$ if 
	\begin{equation}
		\frac{\sigma^2-a^2}{2}\mathbb{E}\big[f''(W^A)\big]=\mathbb{E}\big[f(W)-f(0)-af'(W)\big],\label{ce1}
	\end{equation}
	for all functions $f:\mathbb{R}\rightarrow\mathbb{R}$ for which $f$ and $f'$ are locally absolutely continuous with $\mathbb{E}|f(W)|$, $\mathbb{E}|f'(W)|<\infty.$ We will call the map $W\rightarrow W^A$ the \emph{asymmetric equilibrium transformation}.
\end{definition}

In the case where $\mathbb{E}[W]=0$ and $W$ has positive variance our definition coincides with the centered equilibrium transformation of \cite{pike}, which they name by analogy with the equilibrium transformation which \cite{pekoz1} use in the setting of exponential approximation. It is noted by \cite{pike} that there is a close relationship between their centered equilibrium transformation and the zero bias transformation first defined by \cite{gr97} and used extensively in normal approximation; see, for example, Theorem 3.2 of \cite{pike}. We will further explore the relationship between zero biasing and our asymmetric equilibrium transformation in Section \ref{sec3.2} below. 

The moment condition $\mathrm{Var}(W)>(\mathbb{E}[W])^2$ in Definition \ref{def1} is natural in the setting of asymmetric Laplace approximation. Our interest is in random variables that are close in distribution to the asymmetric Laplace, which may therefore be expected to satisfy this relationship, as does the asymmetric Laplace distribution itself. Such moment conditions also appear in other such definitions. For example, the definition of zero biasing applies only to zero mean random variables.

In the following proposition, we state some useful properties of the asymmetric equilibrium distribution. 

\begin{proposition}\label{prop3.3} Let $W$ be a random variable with mean $a\in\mathbb{R}$ and variance $\sigma^2\in(0,\infty)$ and suppose that $\sigma^2>a^2$. Suppose that $W^A$ has the $W$-asymmetric equilibrium distribution. 

\vspace{3mm}
	
\noindent	(i) The $\mathrm{AL}(0,a,\sqrt{\sigma^2-a^2})$ distribution is the unique fixed point of the asymmetric equilibrium transformation.

\vspace{3mm}
	
\noindent(ii) There exists a unique distribution for $W^A$ such that equation (\ref{ce1}) holds. Moreover, the
distribution of $W^A$ is absolutely continuous and its density is given by
	\begin{align}
		f_{W^A}(w)&=\frac{2}{\sigma^2-a^2}\bigg[a\mathbb{P}(W\leq w)-\int_{0}^{1}\mathbb{E}[W\mathbb{I}(W\leq w/v)]\,\mathrm{d}v\bigg]\label{equality2}\\
        &=\frac{2}{\sigma^2-a^2}\bigg[\int_{0}^{1}\mathbb{E}[W\mathbb{I}(W> w/v)]\,\mathrm{d}v-a\mathbb{P}(W>w)\bigg],\quad w\in\mathbb{R}. \label{equality1}		
        \end{align}
(iii)  Suppose that $\alpha=\inf\mathrm{support}(W)$ and $\beta=\sup\mathrm{support}(W)$ are finite. Then $\mathrm{support}(W^A)=[\alpha,\beta]$. A closed end point in the interval is replaced by an open end point if any of the values $\alpha$ or $\beta$ are infinite. 

In particular, if $|W|\leq C$ for some $C>0$, then $|W^A|\leq C$.

\vspace{3mm
}
  \noindent
	(iv) For $c\in\mathbb{R}$, the random variable $cW^A$ has the $cW$-asymmetric equilibrium distribution.
    \vspace{3mm}

 \vspace{3mm}

    \noindent
(v) For $r\geq0$,
	\begin{align}
		\mathbb{E}[(W^A)^r]=\frac{2\mathbb{E}[W^{r+2}]-2a(r+2)\mathbb{E}[W^{r+1}]}{(r+1)(r+2)(\sigma^2-a^2)},\label{mr}
	\end{align}
	and 
	\begin{align}
		\mathbb{E}[|W^A|^r]=\frac{2\mathbb{E}[|W|^{r+2}]-2a(r+2)\mathbb{E}[|W|^{r+1}\mathrm{sgn}(W)]}{(r+1)(r+2)(\sigma^2-a^2)},\label{absm}
	\end{align}
    where $\mathrm{sgn}(x)=-1$ for $x<0$, $\mathrm{sgn}(0)=0$ and $\mathrm{sgn}(x)=1$ for $x>0$.

    \vspace{3mm}

    \noindent
    (vi) If $W$ has characteristic function $\phi_W(t)=\mathbb{E}[\mathrm{e}^{\mathrm{i}tW}]$, then $W^A$ has characteristic function 
    \begin{equation}\label{eq:trans1}
    \phi_{W^A}(t)=\frac{2(1+[a\mathrm{i}t-1]\phi_W(t))}{t^2(\sigma^2-a^2)}.
    \end{equation}
\end{proposition}

\begin{remark}
Some of the properties of the asymmetric equilibrium distribution given in Proposition \ref{prop3.3} closely resemble those of the zero bias distribution (as given in \cite{gr97} and \cite[Section 2.3.3]{chen}). Indeed, the $\mathrm{N}(0,\sigma^2)$ distribution is the unique fixed point of the zero bias transformation; when a random variable $X$ has zero mean and finite positive variance $\sigma^2$ there exists a unique distribution for the zero biased random variable $X^z$ with absolutely continuous density; if $|X|$ is bounded by some constant $C$ then also $|X^z|\leq C$; and $cX^z$ has the $cX$-zero bias distribution.
\end{remark}

\begin{remark}
Let $k\geq1$ be a positive integer. By H\"older's inequality, we have that $\mathbb{E}[|W|^{k+1}]\leq\mathbb{E}[|W|^{k+2}]/\sqrt{\mathbb{E}[W^2]}\leq \mathbb{E}[|W|^{k+2}]/\sigma$, since $\mathbb{E}[W^2]=\sigma^2+a^2\geq \sigma^2$. Therefore, from the absolute moment relation (\ref{absm}) we obtain the bound
\begin{align}\label{simplebd}
\mathbb{E}[|W^A|^k]&\leq\frac{2\mathbb{E}[|W|^{k+2}]+2(k+2)|a|\mathbb{E}[|W|^{k+1}]}{(k+1)(k+2)(\sigma^2-a^2)}\nonumber \\
&\leq \frac{2(1+(k+2)|a|/\sigma)}{(k+1)(k+2)(\sigma^2-a^2)}\mathbb{E}[|W|^{k+2}].
\end{align}
We will make use of this bound in our proof of Theorem \ref{thm4.1}.
\end{remark}

\begin{proof}
	(i) This follows immediately from Definition \ref{def1} and Lemma \ref{lem2.1}.
	
\vspace{3mm}

\noindent (ii) We begin by verifying the equality between expressions (\ref{equality2}) and (\ref{equality1}). For this, we note that $\int_{0}^{1}\mathbb{E}[W\mathbb{I}(W\leq w/v)]\,\mathrm{d}v+\int_{0}^{1}\mathbb{E}[W\mathbb{I}(W> w/v)]\,\mathrm{d}v=\int_0^1\mathbb{E}[W]\,\mathrm{d}v=a$. Therefore 
\begin{align*}
\int_{0}^{1}\mathbb{E}[W\mathbb{I}(W> w/v)]\,\mathrm{d}v-a\mathbb{P}(W>w)&=a-   \int_{0}^{1}\mathbb{E}[W\mathbb{I}(W\leq w/v)]\,\mathrm{d}v-a(1-\mathbb{P}(W\leq a))\\
&=a\mathbb{P}(W\leq w)-\int_{0}^{1}\mathbb{E}[W\mathbb{I}(W\leq w/v)]\,\mathrm{d}v,
\end{align*}
which confirms the equality between (\ref{equality2}) and (\ref{equality1}).

We note that the distribution of $W^A$ must be unique (if it exists). This follows because $\mathbb{E}[f''(W^A)]=\mathbb{E}[f''(Y^A)]$ for, say, all twice differentiable functions $f$ with compact support implies that $W^A$ and $ Y^A$ are equal in distribution.  

We now prove that, under the assumption $\sigma^2>a^2$, if the distribution of $W^A$ exists then it is absolutely continuous with density $f_{W^A}(w)$, $w\in\mathbb{R}$. We begin by proving that $f_{W^A}(w)\geq0$ for all $w\in\mathbb{R}$. To show this, we will work with the representation (\ref{equality1}). We observe that, for $w\in\mathbb{R}$,
\begin{align*}
a\mathbb{P}(W<w)\leq a=\mathbb{E}[W]=\int_0^1\mathbb{E}[W]\,\mathrm{d}v\leq\int_{0}^{1}\mathbb{E}[W\mathbb{I}(W> w/v)]\,\mathrm{d}v. 
\end{align*}
With this inequality and the assumption $\sigma^2>a^2$ we deduce that $f_{W^A}(w)\geq0$ for all $w\in\mathbb{R}$.

We now prove that $\int^{\infty}_{-\infty} f_{W^A}(w)\,\mathrm{d}w=1$. By an application of Fubini's theorem we obtain that
	\begin{align*}
		\int_{0}^{\infty}\mathbb{E}[\mathbb{I}(W>w)]\,\mathrm{d}w&=\mathbb{E}\bigg[\int_{0}^{\infty}\mathbb{I}(W>w)\,\mathrm{d}w\bigg]=\mathbb{E}\bigg[\int_{0}^{0\vee W}1\,\mathrm{d}w\bigg]=\mathbb{E}\big[W\mathbb{I}(W>0)\big]
        \end{align*}
        and
        \begin{align*}
		\int_{0}^{\infty}\int_{0}^{1}\mathbb{E}\big[W\mathbb{I}(W>w/v)\big]\,\mathrm{d}v\,\mathrm{d}w&=\mathbb{E}\bigg[W\int_{0}^{1}\int_{0}^{\infty}\mathbb{I}(vW>w)\,\mathrm{d}w\,\mathrm{d}v\bigg]\\
		&=\mathbb{E}\bigg[W\int_{0}^{1}vW\mathbb{I}(W>0)\,\mathrm{d}v\bigg]=
		\frac{1}{2}\mathbb{E}\big[W^2\mathbb{I}(W>0)\big].
	\end{align*}
By similar calculations we obtain the equalities
\begin{align*} \int_{-\infty}^{0}\mathbb{E}[\mathbb{I}(W\leq w)]\,\mathrm{d}w&=-\mathbb{E}\big[W\mathbb{I}(W\leq0)\big]
\end{align*}
and
\begin{align*}
 \int^{0}_{-\infty}\int_{0}^{1}\mathbb{E}\big[W\mathbb{I}(W\leq w/v)\big]\,\mathrm{d}v\,\mathrm{d}w
		&=-\frac{1}{2}\mathbb{E}\big[W^2\mathbb{I}(W\leq0)\big].   
\end{align*}
On writing $\int_{-\infty}^\infty f_{W^A}(w)\,\mathrm{d}w=\int_{-\infty}^{0}f_{W^A}(w)\,\mathrm{d}w+\int_{0}^{\infty}f_{W^A}(w)\,\mathrm{d}w$ and using the representations (\ref{equality2}) and (\ref{equality1}) for $f_{W^A}(w)$ on the intervals $(-\infty,0)$ and $(0,\infty)$, respectively, we obtain that
	\begin{align*}
		\int_{-\infty}^\infty f_{W^A}(w)\,\mathrm{d}w&=
		\frac{2}{\sigma^2-a^2}\bigg\{\int_{-\infty}^{0}\bigg(a\mathbb{E}\big[\mathbb{I}(W\leq w)\big]-\int_{0}^{1}\mathbb{E}\big[W\mathbb{I}(W<w/v)\big]\,\mathrm{d}v\bigg)\,\mathrm{d}w\\
        &\quad+\int_{0}^{\infty}\bigg(\int_{0}^{1}\mathbb{E}\big[W\mathbb{I}(W>w/v)\big]\,\mathrm{d}v-a\mathbb{E}[\mathbb{I}(W>w)]\bigg)\,\mathrm{d}w\bigg\}\\
		&=\frac{2}{\sigma^2-a^2}\bigg\{-a\mathbb{E}\big[W\mathbb{I}(W\leq0)\big]+\frac{1}{2}\mathbb{E}\big[W^2\mathbb{I}(W\leq0)\big]\\
		&\quad+\frac{1}{2}\mathbb{E}\big[W^2\mathbb{I}(W>0)\big]-a\mathbb{E}\big[W\mathbb{I}(W>0)\big]\bigg\}\\
		&=\frac{1}{\sigma^2-a^2}\big(\mathbb{E}[W^2]-2a\mathbb{E}[W]\big)=1,
	\end{align*}
where in obtaining the final equality we used that $\mathbb{E}[W]=a$ and $\mathbb{E}[W^2]=\sigma^2+a^2$. We have thus shown that the distribution of $W^A$ is absolutely continuous with density $f_{W^A}(w)$, $w\in\mathbb{R}$.

    We finally prove that if the random variable $W^A$ has density $f_{W^A}(w)$ (as given by equations (\ref{equality2}) and (\ref{equality1})), then $W^A$ satisfies equation (\ref{ce1}).
    We begin by applying Fubini's theorem to obtain several useful equalities. Firstly, we have that
	\begin{align*}
		\int_{0}^{\infty}f''(w)\mathbb{E}\big[\mathbb{I}(W>w)\big]\,\mathrm{d}w&=\mathbb{E}\bigg[\int_{0}^{\infty}f''(w)\mathbb{I}(W>w)\,\mathrm{d}w\bigg]\\
		&=\mathbb{E}\bigg[\int_{0}^{0\vee W}f''(w)\,\mathrm{d}w\bigg]\\
        &=\mathbb{E}\big[f'(0\vee W)-f'(0)\big]\\
        &=\mathbb{E}\big[f'(W)\mathbb{I}(W>0)\big]
\end{align*}
and
\begin{align*}
		\int_{0}^{\infty}f''(w)\bigg(\int_{0}^{1}\mathbb{E}\big[W\mathbb{I}(W>w/v)\big]\,\mathrm{d}v\bigg)\,\mathrm{d}w&=\mathbb{E}\bigg[W\int_{0}^{1}f''(w)\mathbb{I}(W>w/v)\,\mathrm{d}w\,\mathrm{d}v\bigg]\\
		&=\mathbb{E}\bigg[W\int_{0}^{1}\int_{0}^{vW\vee0}f''(w)\,\mathrm{d}w\,
		\mathrm{d}v\bigg]\\	
        &=\mathbb{E}\bigg[\int_{0}^{1}\big(f'(vW\vee 0)-f'(0)\big)\,\mathrm{d}v\bigg]\\
		&=\mathbb{E}\bigg[\int_{0}^{1}Wf'(vW)\mathbb{I}(W>0)\,\mathrm{d}v\bigg]\\
		&=\mathbb{E}\big[\big(f(W)-f(0)\big)\mathbb{I}(W>0)\big].
	\end{align*}
By similar calculations we obtain the equalities
	\begin{align*}
		\int^{0}_{-\infty}f''(w)\mathbb{E}\big[\mathbb{I}(W\leq w)\big]\,\mathrm{d}w&=-\mathbb{E}\big[f'(W)\mathbb{I}(W\leq0)\big]
        \end{align*}
        and
        \begin{align*}
		\int^{0}_{-\infty}f''(w)\bigg(\int_{0}^{1}\mathbb{E}\big[W\mathbb{I}(W>w/v)\big]\,\mathrm{d}v\bigg)\,\mathrm{d}w&=-\mathbb{E}\big[\big(f(W)-f(0)\big)\mathbb{I}(W\leq0)\big].
	\end{align*}
	Combining the above equalities we have 
	\begin{align*}
		\frac{\sigma^2-a^2}{2}\mathbb{E}\big[f''(W^A)\big]&=\int_{-\infty}^{0}f''(w)\bigg(a\mathbb{E}\big[\mathbb{I}(W\leq w)\big]-\int_{0}^{1}\mathbb{E}\big[W\mathbb{I}(W\leq w/v)\big]\,\mathrm{d}v\bigg)\,\mathrm{d}w\\
       & \quad+\int_{0}^{\infty}f''(x)\bigg(\int_{0}^{1}\mathbb{E}\big[W\mathbb{I}(W>w/v)\big]\,\mathrm{d}v-a\mathbb{E}\big[\mathbb{I}(W>w)\big]\bigg)\mathrm{d}w\\
		&=-a\mathbb{E}\big[f'(W)\mathbb{I}(W\leq0)\big]+\mathbb{E}\big[\big(f(W)-f(0)\big)\mathbb{I}(W\leq0)\big]\\
        &\quad+\big(f(W)-f(0)\big)\mathbb{I}(W>0)\big]-a\mathbb{E}[f'(W)\mathbb{I}(W>0)]\\
		&=\mathbb{E}\big[f(W)-f(0)-af'(W)\big],
	\end{align*}
	which verifies that if $W^A$ has density $f_{W^A}$ then equation (\ref{ce1}) is satisfied.
	
\vspace{3mm}

\noindent {(iii)} This follows immediately from the formulas (\ref{equality2}) and (\ref{equality1}) for $f_{W^A}(w)$.

        \vspace{3mm}

\noindent (iv) 
On letting $\tilde{f}(x)=f(cx)$ (so that $\tilde{f}^{(k)}(x)=c^{k}f(cx)$) we get that 
	\begin{align*}
		\mathbb{E}\big[f(cW)-f(0)-acf'(cW)\big]&=\mathbb{E}\big[\tilde{f}(W)-\tilde{f}(0)-a\tilde{f}'(W)\big]\\
		&=\frac{\sigma^2-a^2}{2}\mathbb{E}\big[\tilde{f}''(W^A)\big]=\frac{(c\sigma )^2-(ca)^2}{2}\mathbb{E}[f''(cW^A)].
	\end{align*}
This confirms that $cW^A$ has the $cW$-asymmetric equilibrium distribution.

\vspace{3mm}

\noindent (v) We obtain the moment relation (\ref{mr}) by substituting $f(w)=w^{r+2}$ into (\ref{ce1}), and similarly we obtain the absolute moment relation (\ref{absm}) by plugging $f(w)=|w|^{r+2}$ into (\ref{ce1}).

\vspace{3mm}

\noindent (vi)
By the definition \eqref{ce1} of $W^A$ we have that
\[
\mathbb{E}[f^{\prime\prime}(W^A)]=\frac{2}{\sigma^2-a^2}\left(\mathbb{E}[f(W)]-f(0)-a\mathbb{E}[f^\prime(W)]\right).
\]
Taking $f(x)=-t^{-2}\mathrm{e}^{\mathrm{i}tx}$, the left-hand side above becomes the characteristic function $\phi_{W^A}(t)$ of $W^A$. With this choice, we thus have
\begin{equation*}
\phi_{W^A}(t)=\frac{2}{\sigma^2-a^2}\left(-\frac{1}{t^2}\phi_W(t)+\frac{1}{t^2}+\frac{a\mathrm{i}}{t}\phi_W(t)\right)=\frac{2\left([a\mathrm{i}t-1]\phi_W(t)+1\right)}{t^2(\sigma^2-a^2)},
\end{equation*}
as required.
\end{proof}

\subsection{Construction of asymmetric equilibrium distributions}\label{sec3.2}

In this section, we explore situations in which we may give explicit constructions of the asymmetric equilibrium transformation of a given random variable $W$. The most fruitful of these is the setting where $W$ is a random sum, that is, has the form 
\begin{equation}\label{eq:RandomSum1}
W=\sum_{i=1}^NY_i,
\end{equation}
for some sequence $Y_1,Y_2,\ldots$ of independent and identically distributed random variables which are also independent of the non-negative, integer-valued random variable $N$. In the important special case where $N\sim\mathrm{Geom}(p)$ has a geometric distribution with parameter $p\in(0,1)$ and mass function $\mathbb{P}(N=k)=p(1-p)^{k-1}$ for $k=1,2,\ldots$ we refer to this random sum as a geometric random sum. 

As we will see from the bounds of Theorem \ref{thm3.5} below, we will obtain good asymmetric Laplace approximation results when the difference $|W-W^A|$ is small. Motivated by similar constructions from other equilibrium transformations, in the case of a random sum $W$ as above, we will look for representations of $W^A$ of the form 
\begin{equation}\label{eq:RandomSum2}
W^A=\sum_{i=1}^MY_i+X^A,
\end{equation}
for some suitable random variables $M$ and $X$, and where all random variables on the right-hand side are independent. We will verify such a representation using a characteristic function approach. For a random variable $W$ we will denote its characteristic function by $\phi_W(t)=\mathbb{E}[\mathrm{e}^{\mathrm{i}tW}]$, and for an integer-valued random variable $N$ we let $G_N(z)=\mathbb{E}[z^N]$ denote its probability generating function. Letting $a=\mathbb{E}[W]$ and $\sigma^2=\mathrm{Var}(W)$, and assuming as usual that $\sigma^2>a^2$, we note that for $W$ as in \eqref{eq:RandomSum1} we have $\phi_W(t)=G_N(\phi_Y(t))$, and hence \eqref{eq:trans1} gives us that
\[
\phi_{W^A}(t)=\frac{2\{1+[a\mathrm{i}t-1]G_N(\phi_Y(t))\}}{t^2(\sigma^2-a^2)}.
\]
The random variable on the right-hand side of \eqref{eq:RandomSum2} has characteristic function
\[
G_M(\phi_Y(t))\phi_{X^A}(t)=\frac{2G_M(\phi_Y(t))\{1+[\mathrm{i}t\mathbb{E}[X]-1]\phi_X(t)\}}{t^2(\mathrm{Var}(X)-(\mathbb{E}[X])^2)}.
\]
It follows that to verify the representation \eqref{eq:RandomSum2} for $W$ as in \eqref{eq:RandomSum1}, we need only check that
\begin{equation}\label{eq:CharFn3}
\frac{1+[a\mathrm{i}t-1]G_N(\phi_Y(t))}{(\sigma^2-a^2)G_M(\phi_Y(t))}=\frac{1+[\mathrm{i}t\mathbb{E}[X]-1]\phi_X(t)}{\mathrm{Var}(X)-(\mathbb{E}[X])^2}.
\end{equation}
This is straightforward to complete in the setting where $W$ is a random sum as in \eqref{eq:RandomSum1} with the number of terms $N$ itself a geometric random sum, as in the following lemma.
\begin{lemma}\label{lem:GeometricSum}
Let $Y,Y_1,Y_2,\ldots$ and $T,T_1,T_2,\ldots$ be two independent sequences of independent and identically distributed random variables, where $Y$ is supported on $\mathbb{R}$ and $T$ on the non-negative integers. Let $W=\sum_{j=1}^NY_j$, where $N=\sum_{k=1}^ST_k$ and $S\sim\mathrm{Geom}(q)$ independent of all else. Let $a=\mathbb{E}[W]$ and $\sigma^2=\mathrm{Var}(W)$, and assume that $\sigma^2>a^2$. Then 
\[W^A=\sum_{j=1}^MY_j+X^A\]
has the asymmetric equilibrium distribution with respect to $W$, where $M=\sum_{k=1}^{S-1}T_k$ and $X=\sum_{j=1}^TY_j$.
\end{lemma}
\begin{proof}
It suffices to show that the random variables defined in the lemma satisfy \eqref{eq:CharFn3}. To that end, we note that
\[
G_N(z)=G_S(G_T(z))=\frac{qG_T(z)}{1-(1-q)G_T(z)},
\]
$G_M(z)=q/(1-(1-q)G_T(z))$ and $\phi_X(t)=G_T(\phi_Y(t))$, so that \eqref{eq:CharFn3} becomes
\begin{equation}\label{eq:CharFn4}
\frac{[\mathrm{i}t\mathbb{E}[X]-1]G_T(\phi_Y(t))+1}{\mathrm{Var}(X)-(\mathbb{E}[X])^2}=\frac{[aq\mathrm{i}t-1]G_T(\phi_Y(t))+1}{q(\sigma^2-a^2)}.
\end{equation}
It remains only to check that the relevant moments are such that this equation holds. We have that
\[
a=\mathbb{E}[W]=\mathbb{E}[N]\mathbb{E}[Y]=\frac{\mathbb{E}[T]\mathbb{E}[Y]}{q},
\]
so that 
\begin{equation}\label{eq:Moments1}
\mathbb{E}[X]=\mathbb{E}[T]\mathbb{E}[Y]=aq.
\end{equation}
For the second moments, we have that
\begin{align*}
\sigma^2&=\mathrm{Var}(W)
=\mathrm{Var}(N)(\mathbb{E}[Y])^2+\mathbb{E}[N]\mathrm{Var}(Y)\\
&=(\mathbb{E}[Y])^2\left(\mathrm{Var}(S)(\mathbb{E}[T])^2+\mathbb{E}[S]\mathrm{Var}(T)\right)+\mathbb{E}[N]\mathrm{Var}(Y)\\
&=(\mathbb{E}[Y])^2\left(\frac{(1-q)(\mathbb{E}[T])^2}{q^2}+\frac{\mathrm{Var}(T)}{q}\right)+\frac{\mathbb{E}[T]\mathrm{Var}(Y)}{q},
\end{align*}
so that
\[
q\sigma^2=\mathrm{Var}(T)(\mathbb{E}[Y])^2+\mathbb{E}[T]\mathrm{Var}(Y)+a^2q(1-q).
\]
Then
\[
\mathrm{Var}(X)=\mathrm{Var}(T)(\mathbb{E}[Y])^2+\mathbb{E}[T]\mathrm{Var}(Y)=q\sigma^2-a^2q(1-q)
\]
and
\begin{equation}\label{eq:Moments2}
\mathrm{Var}(X)-(\mathbb{E}[X])^2=q(\sigma^2-a^2).
\end{equation}
From \eqref{eq:Moments1} and \eqref{eq:Moments2} we see that \eqref{eq:CharFn4} does indeed hold, completing the proof.
\end{proof}
Of particular interest for our application in Theorem \ref{thm4.1} is the special case where $W$ is a geometric random sum (i.e., when $T=1$ almost surely). A representation in this case when the $Y_i$ are independent and identically distributed follows immediately from Lemma \ref{lem:GeometricSum}. For our application, however, we wish to relax this assumption and consider independent (but not necessarily identically distributed) random variables. We give the construction of $W^A$ in this setting. Our proof follows the approach used in the proof of Theorem 4.4 of \cite{pike} in which a construction was given for the centered equilibrium distribution of a random sum of independent zero-mean random variables.
\begin{lemma}\label{lem3.4}
	Suppose that $X_1,X_2,\ldots$ is a sequence of independent random variables with $\mathbb{E}[X_i]=0$, $\mathrm{Var}(X_i)=\sigma^2\in(0,\infty)$, $i=1,2,\ldots$.
	Let $Y_i=X_i+\sqrt{p}a$, $i=1,2,\ldots$. Define $W=\sqrt{p}\sum_{i=1}^{N}Y_i$, where $N\sim\mathrm{Geom}(p)$, $p\in(0,1)$, is independent of the $X_i$'s and $a\in\mathbb{R}$ is a constant.  Then, for $\sigma^2>pa^2$, the random variable
	\begin{equation*}
		W^A=\sqrt{p}\bigg(\sum_{i=1}^{N-1}Y_i+Y_N^A\bigg)
	\end{equation*}
	has the asymmetric equilibrium distribution with respect to $W$.
\end{lemma}
\begin{proof} By the law of total probability, we have that 
	\begin{align}
		\mathbb{E}[f''(W^A)] \nonumber
        &=\sum_{n=1}^{\infty}\mathbb{E}\bigg[f''\bigg(\sqrt{p}\sum_{i=1}^{N-1}Y_i+\sqrt{p}Y_N^A\bigg)\,\bigg|\,N=n\bigg]\mathbb{P}(N=n)\nonumber\\
		&=\sum_{n=1}^{\infty}\mathbb{E}\bigg[f''\bigg(\sqrt{p}\sum_{i=1}^{n-1}Y_i+\sqrt{p}Y_n^A\bigg)\bigg]\mathbb{P}(N=n)\nonumber,
        \end{align}
        where we used the independence assumption to obtain the second equality. Let $S_n=\sum_{i=1}^nY_i$ and $f_s(x)=f(\sqrt{p}s+\sqrt{p}x)$. Then we may write
        \begin{align}
    \mathbb{E}\bigg[f''\bigg(\sqrt{p}\sum_{i=1}^{n-1}Y_i+\sqrt{p}Y_n^A\bigg)\bigg]    &=\mathbb{E}\bigg[\mathbb{E}\bigg[f''\bigg(\sqrt{p}s+\sqrt{p}Y_n^A\bigg)\,\bigg|\,S_{n-1}=s\bigg]\bigg]\nonumber\\
        &=\frac{1}{p}\mathbb{E}\big[\mathbb{E}\big[f_s''(Y_n^A)\,|\,S_{n-1}=s\big]\big]. \label{jaja}
	\end{align}
	We also note that
    \begin{align}
		\mathbb{E}[f''(Y^A_n)\,|\,S_{n-1}=s]=	\mathbb{E}[f''(Y_n^A)]&=\frac{2}{\sigma^2-pa^2}\mathbb{E}[f(Y_n)-f(0)-\sqrt{p}af'(Y_n)]\nonumber\\
		&=\frac{2}{\sigma^2-pa^2}\mathbb{E}[f(Y_n)-f(0)-\sqrt{p}af'(Y_n)\,|\,S_{n-1}=s]\label{s1},
	\end{align}
    where we used that $\mathbb{E}[Y_i]=\sqrt{p}a$ and $\mathrm{Var}(Y_i)=\sigma^2$ for $i\geq1$. Inserting (\ref{s1}) into (\ref{jaja}) and using that $Y_n^A$ is independent of $S_{n-1}=\sum_{i=1}^{n-1}Y_i,$ we obtain that
	\begin{align*}
		&\mathbb{E}\bigg[f''\bigg(\sqrt{p}\sum_{i=1}^{n-1}Y_i+\sqrt{p}Y_n^A\bigg)\bigg]\\
		&\quad=\frac{2/p}{\sigma^2-pa^2}\mathbb{E}\bigg[\mathbb{E}\bigg[f_s(Y_n)-f_s(0)-\sqrt{p}af_s'(Y_n)\,\bigg|\,S_{n-1}=s\bigg]\bigg]\\
		&\quad=\frac{2/p}{\sigma^2-pa^2}\mathbb{E}\bigg[\mathbb{E}\bigg[f(\sqrt{p}s+\sqrt{p}Y_n)-f(\sqrt{p}s)-paf'(\sqrt{p}s+\sqrt{p}Y_n)\,\bigg|\,S_{n-1}=s\bigg]\bigg]\\
		&\quad=\frac{2/p}{\sigma^2-pa^2}\mathbb{E}\bigg[f\bigg(\sqrt{p}\sum_{i=1}^{n}Y_i\bigg)-f\bigg(\sqrt{p}\sum_{i=1}^{n-1}Y_i\bigg)-paf'\bigg(\sqrt{p}\sum_{i=1}^{n}Y_i\bigg)\bigg].
	\end{align*}
  To ease notation we will now let $\mathbb{Y}=\{Y_i\}_{i\geq1}$ and $g_n(\mathbb{Y})=f(\sqrt{p}\sum_{i=1}^{n}Y_i)$. Then, since $N\sim \mathrm{Geom}(p)$, we have that $\mathbb{P}(N=n)=p\mathbb{P}(N\geq n)$, and we get that 
	\begin{align*}
		\mathbb{E}[f''(W^A)]		&=\frac{2/p}{\sigma^2-pa^2}\sum_{n=1}^{\infty}\mathbb{E}\bigg[g_n(\mathbb{Y})-g_{n-1}(\mathbb{Y})-paf'\bigg(\sqrt{p}\sum_{i=1}^{n}Y_i\bigg)\bigg]\mathbb{P}(N=n)\\
		&=\frac{2}{\sigma^2-pa^2}\bigg\{\mathbb{E}\bigg[\sum_{n=1}^{\infty}\mathbb{P}(N\geq n)\big(g_n(\mathbb{Y})-g_{n-1}(\mathbb{Y})\big)\bigg]-{a}\mathbb{E}\bigg[f'\bigg(\sqrt{p}\sum_{i=1}^{N}Y_i\bigg)\bigg]\bigg\}\\
		&=\frac{2}{\sigma^2-pa^2}\bigg\{\mathbb{E}\big[g_N(\mathbb{Y})-g_{0}(\mathbb{Y})\big]-{a}\mathbb{E}\bigg[f'\bigg(\sqrt{p}\sum_{i=1}^{N}Y_i\bigg)\bigg]\bigg\}\\
		&=\frac{2}{\mathrm{Var}(W)-(\mathbb{E}[W])^2}\big\{\mathbb{E}[f(W)]-f(0)-\mathbb{E}[W]\mathbb{E}[f'(W)]\big\},
\end{align*}
where in the last step we calculated $\mathbb{E}[W]=\sqrt{p}\mathbb{E}[N]\mathbb{E}[Y_1]=a$ and $\mathrm{Var}(W)=p(\mathbb{E}[N]\mathrm{Var}(Y_1)+\mathrm{Var}(N)(\mathbb{E}[Y_1])^2)=\sigma^2+(1-p)a^2$ by the standard formulas for the mean and variance of a random sum of independent random variables $Y_1=X_1+\sqrt{p}a,Y_2=X_2+\sqrt{p}a,\ldots$ with common mean $\mathbb{E}[Y_1]=\sqrt{p}a$ and common variance $\mathrm{Var}(Y_1)=\sigma^2$. This completes the proof of the lemma.
\end{proof}
\begin{remark}
As we will see, the construction of Lemma \ref{lem3.4} is particularly well suited to application in conjunction with the bounds of Theorem \ref{thm3.5} below. It is therefore natural to look for other situations in which $W$ is a random sum as in \eqref{eq:RandomSum1} for a sequence $Y,Y_1,Y_2,\ldots$ of independent (and, for simplicity here) identically distributed random variables and a non-negative, integer-valued random variable $N$, and we may construct $W^A=\sum_{j=1}^MY_j+Y^A$ for some non-negative, integer-valued random variable $M$. In the case where $W$ is a geometric random sum, this holds with $M=N-1$. In general, if $\sum_{j=1}^{M}Y_j+Y^A$ has the asymmetric equilibrium distribution with respect to $W$, then from \eqref{eq:trans1} we must have
\begin{equation}\label{eq:CharFn2}
\frac{1+[a\mathrm{i}t-1]G_N(\phi_Y(t))}{\sigma^2-a^2}=\frac{G_M(\phi_Y(t))\left\{1+[\mathrm{i}t\mathbb{E}[Y]-1]\phi_Y(t)\right\}}{\mathrm{Var}(Y)-(\mathbb{E}[Y])^2},
\end{equation}
for all $t$, where as usual we write $a=\mathbb{E}[W]$ and $\sigma^2=\mathrm{Var}(W)$. We have that $\mathbb{E}[Y]=a/\mathbb{E}[N]$ and, again using the conditional variance formula,
\[
\sigma^2=\mathrm{Var}(Y)\mathbb{E}[N]+\mathrm{Var(N)}(\mathbb{E}[Y])^2=\mathrm{Var}(Y)\mathbb{E}[N]+\mathrm{Var}(N)\left(\frac{a}{\mathbb{E}[N]}\right)^2,
\]
so that
\[
\mathrm{Var}(Y)=\frac{\sigma^2(\mathbb{E}[N])^2-a^2\mathrm{Var}(N)}{(\mathbb{E}[N])^3},
\]
and \eqref{eq:CharFn2} becomes
\[
\frac{1+[a\mathrm{i}t-1]G_N(\phi_Y(t))}{\sigma^2-a^2}=\frac{G_M(\phi_Y(t))(\mathbb{E}[N])^3}{\sigma^2(\mathbb{E}[N])^2-a^2\mathrm{Var}(N)-a^2\mathbb{E}[N]}\left(1+\left[\frac{a\mathrm{i}t}{\mathbb{E}[N]}-1\right]\phi_Y(t)\right).
\]
That is,
\[
G_M(\phi_Y(t))=\frac{(\sigma^2(\mathbb{E}[N])^2-a^2\mathrm{Var}(N)-a^2\mathbb{E}[N])(1+[a\mathrm{i}t-1]G_N(\phi_Y(t)))}{(\sigma^2-a^2)\mathbb{E}([N])^3\left(1+\left[\frac{a\mathrm{i}t}{\mathbb{E}[N]}-1\right]\phi_Y(t)\right)}.
\]
In order for this to be a valid generating function, the only dependence on $t$ should be through the argument $\phi_Y(t)$ on the left-hand side. That is,
\[
\frac{\partial}{\partial t}\left(\frac{1+[a\mathrm{i}t-1]G_N(z)}{1+\left[\frac{a\mathrm{i}t}{\mathbb{E}[N]}-1\right]z}\right)=0.
\]
Differentiating, we have that this holds if and only if either $a=0$ or
\[
\left(1-z+\frac{z}{\mathbb{E}[N]}\right)G_N(z)=\frac{z}{\mathbb{E}[N]}.
\]
Letting $\mathbb{E}[N]=1/p$, this is equivalent to $G_N(z)=pz/(1-(1-p)z)$, which characterises the geometric distribution. That is, our representation $W^A=\sum_{j=1}^MY_j+Y^A$ holds only if either $N$ is geometric (in which case $M=N-1$) or $a=0$.
\end{remark}


We conclude this section by exploring the relationship between the centered and asymmetric equilibrium transformations, and with the zero bias transformation first defined by \cite{gr97} and which has been extensively used in normal approximation; see \cite{chen} for a discussion of many applications. This will allow us to find further situations beyond the random sum case in which we may explicitly construct $W^A$ and which will be useful for the applications in Section \ref{sec4} below. For a mean-zero random variable $W$, we recall that the random variable $W^z$ has the $W$-zero-biased distribution if
\[
\mathbb{E}[g'(W^z)]=\frac{\mathbb{E}[Wg(W)]}{\mathrm{Var}(W)}\] 
for all functions $g:\mathbb{R}\to\mathbb{R}$ for which the expectation on the right-hand side exists. We will further write $W^L$ for a random variable having the centered equilibrium transformation with respect to $W$. That is, a random variable whose characteristic function satisfies
\begin{equation}\label{eq:trans2}
\phi_{W^L}(t)=\frac{2(1-\phi_W(t))}{t^2\mathrm{Var}(W)},
\end{equation}
using the $\mathbb{E}[W]=0$ case of our \eqref{eq:trans1}. We recall from Theorem 3.2 of \cite{pike} that for a mean-zero random variable $W$ we have $W^L=_dBW^z$, where $B\sim\mathrm{Beta}(2,1)$ has a beta distribution with density function $2x$ for $x\in(0,1)$, independent of $W^z$.
This construction of $W^L$ can also easily be seen using characteristic functions: it is easily checked that $\phi_{W^z}(t)=-\phi^\prime_W(t)/(t\mathrm{Var}(W))$, and so
\[
\phi_{BW^z}(t)=2\int_0^1u\phi_{W^z}(tu)\,\mathrm{d}u=\frac{-2}{t\mathrm{Var}(W)}\int_0^1\phi^\prime_W(tu)\,\mathrm{d}u=\frac{2(1-\phi_W(t))}{t^2\mathrm{Var}(W)},
\]
as required.

Returning to the general case where $\mathbb{E}[W]$ is not necessarily zero, let us suppose that there exists a random variable $\xi$ such that we may write $W=X+\xi$, where $X\sim\mathrm{Exp}(\mathbb{E}[W]^{-1})$ has an exponential distribution with the same mean as $W$ and is independent of $\xi$. We clearly have that $\mathbb{E}[\xi]=0$ and
\[
\mathrm{Var}(\xi)=\mathrm{Var}(W)-\mathrm{Var}(Z)=\mathrm{Var}(W)-(\mathbb{E}[W])^2.
\]
It is thus necessary that $\mathrm{Var}(W)\geq(\mathbb{E}[W])^2$ for the existence of such a random variable $\xi$. Then since
\[
\phi_W(t)=\phi_X(t)\phi_\xi(t)=\frac{\phi_\xi(t)}{1-\mathrm{i}t\mathbb{E}[W]},
\]
we have $\phi_\xi(t)=(1-\mathrm{i}t\mathbb{E}[W])\phi_W(t)$, and it follows from \eqref{eq:trans1} and \eqref{eq:trans2} that we may write $W^A=_d\xi^L$. In particular, we have $W^A=_dB\xi^z$, where $B\sim\mathrm{Beta}(2,1)$ is independent of $\xi^z$.

We illustrate this discussion with three examples, the third of which we will also use as the basis for the application considered in Section \ref{sec4.0}.
\begin{example}
Let $X\sim\mathrm{Exp}(\lambda)$ have an exponential distribution with mean $\lambda^{-1}$. Then $\xi$ is a point mass at zero. Since this has zero variance, its zero-biased version is not defined. Similarly, since $\mathrm{Var}(X)-(\mathbb{E}[X])^2=0$, the asymmetric equilibrium version $X^A$ is also undefined.   
\end{example}

\begin{example}
Let $W$ be exponential data perturbed by the addition of independent normal noise. That is, $W=X+N$, where $X\sim\mathrm{Exp}(\lambda)$ and $N\sim\mathrm{N}(0,\tau^2)$ are independent. Then from the above, and noting that the mean-zero normal is a fixed point of the zero-bias transformation (see Lemma 2.1(i) of \cite{gr97}), we have $W^A=_dN^L=_dBN^z=_dBN$, where $B\sim\mathrm{Beta}(2,1)$ is independent of $N$.
\end{example}

\begin{example}
Let $Z\sim\mathrm{AL}(0,a,b)$ have an asymmetric Laplace distribution. Then from \eqref{nvm} we may write $Z=_daX+\sqrt{X}N_b$, where $X\sim\mathrm{Exp}(1)$ and $N_b\sim\mathrm{N}(0,b^2)$ are independent. Since $aX\sim\mathrm{Exp}(\mathbb{E}[Z]^{-1})$ and $Z$ is a fixed point of our asymmetric equilibrium transformation, it follows from the discussion above that we may write 
\begin{equation}\label{eq:zerobias}
Z=_dZ^A=_d\left(\sqrt{X}N_b\right)^L=_dB\left(\sqrt{X}N_b\right)^z,
\end{equation}
where $B\sim\mathrm{Beta}(2,1)$ is independent of all else.
\end{example}

\subsection{General bounds using the asymmetric equilibrium transformation}\label{sec3.3}

In the following theorem, we provide general bounds for asymmetric Laplace approximation in terms of the asymmetric equilibrium transformation. 

\begin{theorem}\label{thm3.5}
	Let $W$ be a random variable with mean $a\in\mathbb{R}$ and variance $\sigma^2\in(0,\infty)$, and suppose that $\sigma^2>a^2$. Let $W^A$ have the $W$-asymmetric equilibrium distribution and let $Z\sim\mathrm{AL}(0,a,\sqrt{\sigma^2-a^2})$. 
    
\vspace{3mm}

\noindent (i) For any $\beta>0$,
	\begin{align}
		d_\mathrm{K}(W,Z)&\leq \bigg(\frac{14}{\sqrt{2\sigma^2-a^2}}+\frac{7 |a|}{\sigma^2-a^2}\bigg)\beta+\bigg(5+\frac{7|a|}{\sqrt{2\sigma^2-a^2}}\bigg)\mathbb{P}(|W-W^A|>\beta)\label{3.51},	 \\
		\label{3.52}		d_\mathrm{K}(W^A,Z)&\leq \bigg(\frac{2}{\sqrt{2\sigma^2-a^2}}+\frac{4|a|}{\sigma^2-a^2}\bigg)\beta+\bigg(2+\frac{4|a|}{\sqrt{2\sigma^2-a^2}}\bigg)\mathbb{P}\big(|W-W^A|>\beta\big).
	\end{align}
 Further assume that $\mathbb{E}[|W|^{k+2}]<\infty$ for some $k\geq1$. Then 
	\begin{align}
		d_\mathrm{K}(W,Z)&\leq	
		\bigg(\frac{14}{\sqrt{2\sigma^2-a^2}}+\frac{7 |a|}{\sigma^2-a^2}\bigg)^\frac{k}{k+1}\nonumber\\
		&\quad\times\bigg\{(k+k^{-k})\bigg(5+\frac{7|a|}{\sqrt{2\sigma^2-a^2}}\bigg)\mathbb{E}\big[|W-W^A|^k\big]\bigg\}^\frac{1}{k+1}.\label{kolb}
	\end{align}  
\noindent (ii) Suppose now that $\mathbb{E}\big[|W|^3\big]<\infty$. Then
	\begin{align}
		\label{3.53}		d_\mathrm{W}(W,Z)&\leq 2\mathbb{E}|W-W^A|,\\
		\label{3.54}		d_\mathrm{W}(W^A,Z)&\leq \bigg(1+\frac{2}{\sqrt{2\sigma^2-a^2}}\bigg)\mathbb{E}|W-W^A|,\\
		\label{3.55}		d_\mathrm{K}(W^A,Z)&\leq\bigg(\frac{2}{\sqrt{2\sigma^2-a^2}}+\frac{4}{\sigma^2-a^2}\bigg)\mathbb{E}|W-W^A|.
	\end{align}
\noindent (iii) Suppose that $\mathbb{E}[W^4]<\infty.$ Then
	\begin{align}
		d_2(W,Z)\leq\frac{\sigma^2-a^2}{\sqrt{2\sigma^2-a^2}}\mathbb{E}\big[|\mathbb{E}[W-W^A\,|\,W
		]|\big]+\mathbb{E}\big[(W-W^A)^2\big].\label{3.56}
	\end{align}
\end{theorem}

\begin{proof} (i) For $c\in\mathbb{R}$ and $\epsilon\geq0$, let $f_{c,\epsilon}$ be the solution (\ref{sol}) to the $\mathrm{AL}(0,a,\sqrt{\sigma^2-a^2})$ Stein equation with test function $h_{c,\epsilon}$, which is defined by $h_{c,\epsilon}(x):=\epsilon^{-1}\int_{0}^{\epsilon}\mathbb{I}(x+s\leq c)\,\mathrm{d}s$, for $\epsilon>0$, and  $h_{c,0}(x)=\mathbb{I}(x\leq c)$. From the estimates of Lemma \ref{lem2.2} we have the bounds
	\begin{align}
		\label{fcb1}	\|f_{c,\epsilon}\|&\leq1,\\
		\label{fcb2}	\|f'_{c,\epsilon}\|&\leq \frac{2}{\sqrt{2\sigma^2-a^2}},\\
		\label{fcb3}	\|f''_{c,\epsilon}\|&\leq\frac{4}{\sigma^2-a^2}.
	\end{align}
To simplify notation, we will write $f:=f_{c,\epsilon}$. We also define $\Delta:=W-W^A$ and $I_1:=\mathbb{I}(|\Delta|<\beta)$.     

By the asymmetric Laplace Stein equation (\ref{ivp}) (with $b=\sqrt{\sigma^2-a^2}$ and using also that $f(0)=0$) we have 
\begin{align*}
	&\mathbb{E}[h_{c,\epsilon}(W)]-\mathbb{E}[h_{c,\epsilon}(Z)]\\
	&\quad=\mathbb{E}\bigg[\frac{\sigma^2-a^2}{2}f''(W)+af'(W)-f(W)\bigg]\\
	&\quad=\frac{\sigma^2-a^2}{2}\mathbb{E}[I_1\big(f''(W)-f''(W^A)\big)]+\frac{\sigma^2-a^2}{2}\mathbb{E}[(1-I_1)\big(f''(W)-f''(W^A)\big)]\\
	&\quad=:J_1+J_2.
\end{align*}
We note here that the expectation $\mathbb{E}[f''(W^A)]$ is well defined, due to the bound $(\ref{fcb3})$.

Letting $\tilde{h}(x)=h(x)-\mathbb{E}[h(Z)]$ for $Z\sim\mathrm{AL}(0,a,b)$, we will begin by bounding $|J_2|$:
\begin{align*}
	|J_2|&=\big|\mathbb{E}\big[(1-I_1)\big(f(W)-f(W^A)-a(f'(W)-f'(W^A))+\tilde{h}_{c,\epsilon}(W)-\tilde{h}_{c,\epsilon}(W^A)\big)\big]\big|\\
	&\leq\big(2\|f\|+2|a|\|f'\|+1\big)\mathbb{P}(|\Delta|>\beta)\leq\bigg(3+\frac{4|a|}{\sqrt{2\sigma^2-a^2}}\bigg)\mathbb{P}(|\Delta|>\beta),
\end{align*}
where we used the bounds (\ref{fcb1}) and (\ref{fcb2}) to obtain the last inequality. 

To bound $J_1$, we will make use of the following two inequalities, whose simple proofs we omit. Firstly, for $-\infty<\alpha<\gamma<\infty$, we have the bound
\begin{equation}
	\mathbb{P}(\alpha\leq W\leq \gamma)\leq\frac{\gamma-\alpha}{\sqrt{2\sigma^2-a^2}}+2\kappa,\label{epf1}
\end{equation}
where $\kappa:=d_\mathrm{K}(W,Z)$. Secondly, 
for any $\epsilon>0$,
\begin{equation}
	\kappa=d_\mathrm{K}(W,Z)\leq\frac{\epsilon}{\sqrt{2\sigma^2-a^2}}+\sup_{c\in\mathbb{R}}\big|\mathbb{E}[h_{c,\epsilon}(W)]-\mathbb{E}[h_{c,\epsilon}(Z)]\big|.\label{epf2}
\end{equation}

By differentiating both sides of the Stein equation (\ref{ivp}) with test function $h_{c,\epsilon}$, we have that 
\begin{align*}
	|J_1|&=\bigg|\mathbb{E}\bigg[I_1\int_{0}^{-\Delta}\frac{\sigma^2-a^2}{2}f^{(3)}(W+t)\,\mathrm{d}t\bigg]\bigg|\\
	&=\bigg|\mathbb{E}\bigg[I_1\int_{0}^{-\Delta}\big\{f'(W+t)-af''(W+t)-\frac{1}{\epsilon}\mathbb{I}(c-\epsilon\leq W+t\leq c)\big\}\,\mathrm{d}t\bigg]\bigg|\\
	&\leq\|f'\|\mathbb{E}|I_1\Delta|+|a|\|f''\|\mathbb{E}|I_1\Delta|+\frac{1}{\epsilon}\int_{-\beta}^{0}\mathbb{P}(c-\epsilon\leq W+t\leq c)\,\mathrm{d}t\\
	&\leq\frac{3\beta}{\sqrt{2\sigma^2-a^2}}+\frac{4\beta |a|}{\sigma^2-a^2}+\frac{2\beta\kappa}{\epsilon},
\end{align*}
where we obtained the final inequality by using inequalities (\ref{fcb2}), (\ref{fcb3}) and (\ref{epf1}). Now, applying inequality (\ref{epf2}) and taking $\epsilon=\eta\beta$ for some $\eta>2$ yields
\begin{align*}
	\kappa&\leq\bigg(3+\frac{4|a|}{\sqrt{2\sigma^2-a^2}}\bigg)\mathbb{P}(|\Delta|>\beta)+\frac{3\beta+\epsilon}{\sqrt{2\sigma^2-a^2}}+\frac{4\beta |a|}{\sigma^2-a^2}+\frac{2\beta\kappa}{\epsilon}\\
	&=\bigg(3+\frac{4|a|}{\sqrt{2\sigma^2-a^2}}\bigg)\mathbb{P}(|\Delta|>\beta)+\frac{(3+\eta)\beta}{\sqrt{2\sigma^2-a^2}}+\frac{4\beta |a|}{\sigma^2-a^2}+\frac{2\kappa}{\eta},
\end{align*}
and a rearrangement gives the bound
\begin{equation}
	\kappa\leq\frac{\eta}{\eta-2}\bigg\{\bigg(3+\frac{4|a|}{\sqrt{2\sigma^2-a^2}}\bigg)\mathbb{P}(|\Delta|>\beta)+\frac{4\beta |a|}{\sigma^2-a^2}\bigg\}+\frac{(3\eta+\eta^2)\beta}{\sqrt{2\sigma^2-a^2}(\eta-2)}.\label{kap}
\end{equation}
Setting $\eta=5$ in inequality (\ref{kap}) and rounding constants up to the nearest integer yields the bound (\ref{3.51}).

We now derive inequality (\ref{3.52}). Let $h\in\mathcal{H}_{\mathrm{K}}$ and let $f_h$ be the solution (\ref{sol}) to the $\mathrm{AL}(0,a,\sqrt{\sigma^2-a^2})$ Stein equation with test function $h$. Arguing similarly to before we get that, for any $h\in\mathcal{H}_{\mathrm{K}}$,
\begin{align*}
|\mathbb{E}[h(W^A)]-\mathbb{E}[h(Z)]|
	&=\bigg|\mathbb{E}\bigg[\frac{\sigma^2-a^2}{2}f_h''(W^A)+af_h'(W^A)-f_h(W^A)\bigg]
	\bigg|\\
    \quad&=\big|\mathbb{E}\big[f_h(W)-f_h(W^A)-a\big(f_h'(W)-f_h'(W^A)\big)\big]\big|\\
	\quad&\leq\big|\mathbb{E}\big[I_1\big(f_h(W)-f_h(W^A)-a(f_h'(W)-f_h'(W^A))\big)\big]\big|\\
	&\quad\quad+\big|\mathbb{E}\big[(1-I_1)\big(f_h(W)-f_h(W^A)-a(f_h'(W)-f_h'(W^A))\big)\big]\big| \\
    \quad&\leq\big(\|f_h'\|+|a|\|f_h''\|\big)\mathbb{E}|I_1\Delta|+\big(2\|f_h\|+2|a|\|f_h'\|\big)\mathbb{P}(|\Delta|>\beta)\\
	\quad&\leq\bigg(\frac{2}{\sqrt{2\sigma^2-a^2}}+\frac{4|a|}{\sigma^2-a^2}\bigg)\beta+\bigg(2+\frac{4|a|}{\sqrt{2\sigma^2-a^2}}\bigg)\mathbb{P}(|\Delta|>\beta),
\end{align*}
where in obtaining the last two inequalities we used the mean value theorem and the bounds of (\ref{kolstein}) for $\|f_h\|$, $\|f_h''\|$ and $\|f_h''\|$ (with $\|\tilde{h}\|\leq 1$ and $b=\sqrt{\sigma^2-a^2}$). We have thus proven inequality (\ref{3.52}).

Suppose now that $\mathbb{E}[|W|^{k+2}]<\infty$ for $k\geq1$. By an application of Markov's inequality to the bound (\ref{3.51}) we get that
	\begin{align}
		&d_\mathrm{K}(W,Z)\leq \bigg(\frac{14}{\sqrt{2\sigma^2-a^2}}+\frac{7 |a|}{\sigma^2-a^2}\bigg)\beta+\bigg(5+\frac{7|a|}{\sqrt{2\sigma^2-a^2}}\bigg)\frac{\mathbb{E}[|W-W^A|^k]}{\beta^{k}}.\label{beq}
	\end{align}
	The bound (\ref{beq}) is optimised by taking 
	\begin{align*}
		\beta=\bigg(k\bigg(5+\frac{7|a|}{\sqrt{2\sigma^2-a^2}}\bigg)\bigg(\frac{14}{\sqrt{2\sigma^2-a^2}}+\frac{7 |a|}{\sigma^2-a^2}\bigg)^{-1}\mathbb{E}\big[|W-W^A|^k\big]\bigg)^{\frac{1}{k+1}}.
	\end{align*}
	Substituting this choice of $\beta$ into (\ref{beq}) yields inequality (\ref{kolb}).

\vspace{3mm}

\noindent (ii) We will now suppose that $\mathbb{E}[|W|^3]<\infty$. Observe that this assumption ensures that $\mathbb{E}|W^A|<\infty$ due to the absolute moment relation (\ref{absm}). Then, for $h\in\mathcal{H}_\mathrm{W}$,
\begin{align*}
	|\mathbb{E}[h(W)]-\mathbb{E}[h(Z)]|&=\bigg|\mathbb{E}\bigg[\frac{\sigma^2-a^2}{2}f_h''(W)+af_h'(W)-f_h(W)\bigg]\bigg|\\
	&=\frac{\sigma^2-a^2}{2}\big|\mathbb{E}\big[f_h''(W)-f_h''(W^A)\big]\big|\\
	&\leq\frac{\sigma^2-a^2}{2}\|f_h^{(3)}\|\mathbb{E}|W-W^A|\\
	&\leq2\mathbb{E}|W-W^A|,
\end{align*}
where in the last step we applied the bound $\|f_h^{(3)}\|\leq4\|h'\|/({\sigma^2-a^2})\leq4/({\sigma^2-a^2})$ of Lemma \ref{lem2.2} (since $\|h'\|\leq 1$ for $h\in\mathcal{H}_\mathrm{W}$). We have thus proven inequality (\ref{3.53}). 

We now prove inequality (\ref{3.54}). For $h\in\mathcal{H}_\mathrm{W}$,
\begin{align}
	|\mathbb{E}[h(W)]-\mathbb{E}[h(Z)]|&=\bigg|\mathbb{E}\bigg[\frac{\sigma^2-a^2}{2}f_h''(W^A)+af_h'(W^A)-f_h(W^A)\bigg]\bigg|\nonumber\\
    &=\big|\mathbb{E}\big[f_h(W)-f_h(W^A)-a(f_h'(W)-f_h'(W^A))\big]\big|\nonumber\\
	&\leq\|f_h'\|\mathbb{E}|W-W^A|+|a|\|f_h''\|\mathbb{E}|W-W^A|\label{b1}.
\end{align}
Applying the bounds $\|f_h'\|\leq\|h'\|\leq1$ and $\|f_h''\|\leq2\|h'\|/\sqrt{2\sigma^2-a^2}\leq 2
/\sqrt{2\sigma^2-a^2}$ of Lemma \ref{lem2.2} to inequality (\ref{b1}) yields inequality (\ref{3.54}). We obtain inequality (\ref{3.55}) similarly, this time taking $h\in\mathcal{H}_\mathrm{K}$ and applying the bounds $\|f_h'\|\leq2\|\tilde{h}\|/\sqrt{2\sigma^2-a^2}\leq 2/\sqrt{2\sigma^2-a^2}$ and $\|f_h''\|\leq4\|\tilde{h}\|/(\sigma^2-a^2)\leq4/(\sigma^2-a^2)$ to inequality (\ref{b1}).

\vspace{3mm}

\noindent (iii) Let $h\in\mathcal{H}_2$ and assume that $\mathbb{E}[W^4]<\infty$ (which guarantees that $\mathbb{E}[(W^A)^2]<\infty$ by the moment relation (\ref{mr})). Taylor expanding gives that 
\begin{align}
	|\mathbb{E}[h(W)]-\mathbb{E}[h(Z)]|&=\bigg|\mathbb{E}\bigg[\frac{\sigma^2-a^2}{2}f_h''(W)+af_h'(W)-f_h(W)\bigg]\bigg|\nonumber\\
	&=\frac{\sigma^2-a^2}{2}\big|\mathbb{E}\big[f_h''(W)-f_h''(W^A)\big]\big|\nonumber\\
	&\leq\frac{\sigma^2-a^2}{2}\big|\mathbb{E}\big[f_h^{(3)}(W)(W-W^A)\big]\big|+\frac{\sigma^2-a^2}{4}\|f_h^{(4)}\|\mathbb{E}\big[(W-W^A)^2\big]\nonumber\\
	&\leq\frac{\sigma^2-a^2}{2}\big|\mathbb{E}\big[f_h^{(3)}(W)\mathbb{E}[W-W^A\,|\,W]\big]\big|+\frac{\sigma^2-a^2}{4}\|f_h^{(4)}\|\mathbb{E}\big[(W-W^A)^2\big]\nonumber\\
	&\leq\frac{\sigma^2-a^2}{2}\|f_h^{(3)}\|\mathbb{E}\big[|\mathbb{E}[W-W^A\,|\,W]|\big]+\frac{\sigma^2-a^2}{4}\|f_h^{(4)}\|\mathbb{E}\big[(W-W^A)^2\big].\nonumber
\end{align}
Inequality (\ref{3.56}) now follows from the bounds $\|f_h^{(3)}\|\leq2\|h''\|/\sqrt{2\sigma^2-a^2}\leq2/\sqrt{2\sigma^2-a^2}$ and $\|f_h^{(4)}\|\leq4\|h''\|/(\sigma^2-a^2)\leq4/(\sigma^2-a^2)$ from Lemma $\ref{lem2.2}$.
\end{proof}

\section{Asymmetric Laplace approximation}\label{sec4}

In this section, we employ our preceding constructions to establish explicit approximation bounds in a number of settings where the asymmetric Laplace appears as a limiting distribution. We begin with an
application to a perturbed asymmetric Laplace distribution in Section \ref{sec4.0}, before moving on to consider geometric random sums in Section \ref{sec4.1} and sums of independent random variables with random standardisation in Section \ref{sec4.2}.

\subsection{The normal-asymmetric Laplace distribution and its relatives}\label{sec4.0}

Consider asymmetric Laplace data perturbed by the addition of independent, mean-zero noise $\eta$ with variance $\tau^2$. In the special case where $\eta$ has a normal distribution, we obtain the normal-asymmetric Laplace distribution introduced by \cite{reed}. In the following theorem, we provide general bounds without any distributional assumptions on $\eta$, and in the case $\eta\sim \mathrm{N}(0,\tau^2)$ we obtain improved bounds, in terms of a better rate of convergence with respect to $\tau$ for the Kolmogorov distance bound and a better constant for the Wasserstein distance bound.
\begin{theorem}\label{thm:perturb}
(i) Let $Z^\prime\sim\mathrm{AL}(0,a,b)$ and $W=Z^\prime+\eta$ with $\eta$ as above. Then
\begin{align}
d_\mathrm{K}(W,Z)&\leq1.05\left(\frac{b^2c^2d\tau^2}{b^2+\tau^2}\right)^{1/3}+\frac{d\tau^2}{b^2+\tau^2},\label{4.1k1} \\
d_\mathrm{W}(W,Z)&\leq \tau+\frac{\tau^2}{\sqrt{2b^2+a^2}},\label{4.1w1}
\end{align}
where $Z\sim\mathrm{AL}(0,a,\sqrt{b^2+\tau^2})$, and
\[c=\frac{14}{\sqrt{a^2+2b^2+2\tau^2}}+\frac{7|a|}{b^2+\tau^2}, \quad d=5+\frac{7|a|}{\sqrt{a^2+2b^2+2\tau^2}}.\] 
(ii) Suppose further that for some $n\geq2$ there exists $C>0$ such that 
\begin{equation}\label{eq:tails}
\mathbb{P}\left(|\eta|>s\right)\leq\frac{C\tau^n}{s^n}
\end{equation}
for all $s>0$. Then the order $\tau^{2/3}$ rate in the Kolmogorov distance bound (\ref{4.1k1}) can be improved to order $\tau^{n/(n+1)}$:
\begin{equation}
d_\mathrm{K}(W,Z)\leq
2\left(\frac{2b^2c^nCd\tau^n}{(b^2+\tau^2)(n^2+3n+2)}\right)^{1/(n+1)}+\frac{d\tau^2}{b^2+\tau^2}. \label{4.1k2}
\end{equation}
(iii) Let $0<\tau<1$ and suppose now that $\eta\sim\mathrm{N}(0,\tau^2)$, so that $W$ follows the normal-asymmetric Laplace distribution. Then 
\begin{align}
d_\mathrm{K}(W,Z)&\leq \bigg(2c+\frac{b^2d}{2\sqrt{2\pi}(b^2+\tau^2)\log(1/\tau)}\bigg)\tau\sqrt{\log(1/\tau)}+\frac{d\tau^2}{b^2+\tau^2}, \label{4.1k3}\\
d_\mathrm{W}(W,Z)&\leq \sqrt{\frac{2}{\pi}}\tau+\frac{\tau^2}{\sqrt{2b^2+a^2}}.    \label{dddw}
\end{align}
\end{theorem}
\begin{proof}
(i) With \eqref{nvm} in view, we write
\[
W=_d aX+\sqrt{X}N_b+\eta,
\]
where $X$ has a unit-mean exponential distribution, $N_b\sim\mathrm{N}(0,b^2)$ has a normal distribution, and all random variables are independent. Then, similarly to \eqref{eq:zerobias}, we may write $W^A=d B(\sqrt{X}N_b+\eta)^z$, where $B\sim\mathrm{Beta}(2,1)$ is independent of all else. By Lemma 2.1(v) of \cite{gr97}, we may construct the zero-biased version of a sum of independent, mean-zero random variables by choosing one of the summands, independently of all else, with probability proportional to its variance, and then replacing the chosen summand with a zero-biased version. Noting that $\mathrm{Var}(\sqrt{X}N_b)=b^2$, and letting $I$ be a Bernoulli random variable, independent of all else, with mean $b^2/(b^2+\tau^2)$, we may write
\[
W^A=_d BI\left[\left(\sqrt{X}N_b\right)^z+\eta\right]+B(1-I)\left[\sqrt{X}N_b+\eta^z\right].
\]
Upon using \eqref{eq:zerobias}, we thus have
\[
W^A=_d I(W-\eta)+BI\eta+B(1-I)(\sqrt{X}N_b+\eta^z),
\]
and hence
\begin{equation}\label{eq:diff}
W-W^A=_d (1-I)W+(1-B)I\eta-B(1-I)(\sqrt{X}N_b+\eta^z).
\end{equation}
In order to apply the bound from Theorem \ref{thm3.5}(i), we bound $\mathbb{P}(|W-W^A|>\beta)$ for $\beta>0$. By conditioning on $I$ we have that
\begin{align}
\nonumber\mathbb{P}(|W-W^A|>\beta)&\leq\frac{b^2}{b^2+\tau^2}\mathbb{P}((1-B)|\eta|>\beta)+\frac{\tau^2}{b^2+\tau^2}\\
\label{eq:chebyshev}&=\frac{2b^2}{b^2+\tau^2}\int_0^1u\mathbb{P}\left(|\eta|>\frac{\beta}{1-u}\right)\,\mathrm{d}u+\frac{\tau^2}{b^2+\tau^2}\\
\nonumber&\leq\frac{2b^2\tau^2}{\beta^2(b^2+\tau^2)}\int_0^1u(1-u)^2\,\mathrm{d}u+\frac{\tau^2}{b^2+\tau^2}=\frac{\tau^2}{b^2+\tau^2}\left(\frac{b^2}{6\beta^2}+1\right),
\end{align}
where the final inequality uses Chebyshev's inequality. Then inequality \eqref{3.51} gives
\[
d_\mathrm{K}(W,Z)\leq c\beta+\frac{d\tau^2}{b^2+\tau^2}\left(\frac{b^2}{6\beta^2}+1\right)
\]
for any $\beta>0$. Choosing 
\[\beta=\left(\frac{b^2d\tau^2}{3(b^2+\tau^2)c}\right)^{1/3}\]
and noting that $3^{-1/3}+3^{2/3}/6<1.05$ yields the Kolmogorov distance bound (\ref{4.1k1}).

We could similarly derive a Wasserstein error bound using the representation \eqref{eq:diff} in conjunction with \eqref{3.53}, but note that a bound of the same order can be more simply derived by writing
\begin{equation}\label{eq:wass1}
d_\mathrm{W}(W,Z^\prime)\leq\mathbb{E}|W-Z^\prime|=\mathbb{E}|\eta|\leq\sqrt{\mathbb{E}[\eta^2]}=\tau,
\end{equation}
and combining this with a bound on $d_\mathrm{W}(Z^\prime,Z)$ established using a standard comparison of generators argument: We let $f_h$ denote the solution to the Stein equation \eqref{ivp} for $Z^\prime\sim\mathrm{AL}(0,a,b)$ and $h\in\mathcal{H}_\mathrm{W}$, and note that we may write 
\[
d_\mathrm{W}(Z^\prime,Z)=\sup_{h\in\mathcal{H}_\mathrm{W}}\left|\mathbb{E}\left[\frac{b^2}{2}f_h^{\prime\prime}(Z)+af_h^\prime(Z)-f_h(Z)\right]-\mathbb{E}\left[\frac{b^2+\tau^2}{2}f_h^{\prime\prime}(Z)+af_h^\prime(Z)-f_h(Z)\right]\right|\,,
\]
since this latter expectation is zero by Lemma \ref{lem2.1}. Hence
\begin{equation}\label{eq:wass2}
d_\mathrm{W}(Z^\prime,Z)\leq\frac{\tau^2}{2}\sup_{h\in\mathcal{H}_\mathrm{W}}\lVert f_h^{\prime\prime}\rVert\leq\frac{\tau^2}{\sqrt{2b^2+a^2}},
\end{equation}
by \eqref{2.6}. The result \eqref{4.1w1} now follows on combining \eqref{eq:wass1} and \eqref{eq:wass2} using the triangle inequality. 

\vspace{3mm}

\noindent (ii) We follow the proof of part (i) up to \eqref{eq:chebyshev}, where we replace the use of Chebyshev's inequality with \eqref{eq:tails} to get
\begin{align*}
\mathbb{P}(|W-W^A|>\beta)&\leq\frac{2b^2C\tau^n}{\beta^n(b^2+\tau^2)}\int_0^1u(1-u)^n\,\mathrm{d}u+\frac{\tau^2}{b^2+\tau^2}\\
&=\frac{2b^2C\tau^n}{\beta^n(b^2+\tau^2)(n^2+3n+2)}+\frac{\tau^2}{b^2+\tau^2}.
\end{align*}
We now apply inequality \eqref{3.51}, choose 
\[
\beta=\left(\frac{2b^2Cdn\tau^n}{c(b^2+\tau^2)(n^2+3n+2)}\right)^{1/(n+1)},
\]
and note that $n^{1/(n+1)}+n^{-n/(n+1)}\leq2$ to obtain the bound (\ref{4.1k2}).

\vspace{3mm}

\noindent (iii) We start by recalling the classical Mill's ratio upper bound for the standard normal distribution. Let $N$ follow the standard normal distribution. Then $\mathbb{P}(N>x)<\varphi(x)/x$ for $x>0$, where $\varphi(x)$ denotes the standard normal density. We follow the proof of part (i) up to (\ref{eq:chebyshev}) but now use the Mill's ratio upper bound to get
\begin{align*}
\mathbb{P}(|W-W^A|>\beta)&\leq\frac{2b^2}{b^2+\tau^2}\cdot\frac{1}{\beta}\sqrt{\frac{2}{\pi}}\int_0^1u(1-u)\exp\bigg(-\frac{\beta^2}{2\tau^2(1-u)^2}\bigg)\,\mathrm{d}u +\frac{\tau^2}{b^2+\tau^2}\\
&\leq\frac{2b^2}{b^2+\tau^2}\cdot\frac{1}{\beta}\sqrt{\frac{2}{\pi}}\int_0^1\frac{1}{4}\exp\bigg(-\frac{\beta^2}{2\tau^2}\bigg)\,\mathrm{d}u +\frac{\tau^2}{b^2+\tau^2}\\
&=\frac{b^2}{\sqrt{2\pi}\beta(b^2+\tau^2)}\exp\bigg(-\frac{\beta^2}{2\tau^2}\bigg) +\frac{\tau^2}{b^2+\tau^2}.
\end{align*}
Applying inequality (\ref{3.51}) with $\beta=2\tau\sqrt{\log(1/\tau)}$ now yields the bound (\ref{4.1k3}). Note that our assumption that $0<\tau<1$ ensures that $\beta>0$.

Finally, we note the improved Wasserstein distance bound (\ref{dddw}) follows since $\mathbb{E}|\eta|=\tau\sqrt{2/\pi}$ when $\eta\sim \mathrm{N}(0,\tau^2)$ (which improves on the estimate $\mathbb{E}|\eta|\leq \tau$ used in part (i) of the proof).
\end{proof}

\subsection{Geometric random sums}\label{sec4.1}

Let $X_1,X_2,\ldots$ be a sequence of i.i.d.\ positive, non-degenerate random variables with mean $\lambda>0$ and let $N_p\sim \mathrm{Geom}(p)$ be independent of the $X_i$'s. A celebrated result of \cite{renyi} states that the geometric random sum $p\sum_{i=1}^{N_p} X_i$ converges in distribution as $p\rightarrow0$ to an exponential random variable with mean $1/\lambda$. In \cite{pekoz1}, the authors applied Stein's method for exponential approximation with the equilibrium coupling to derive explicit $O(p^{1/2})$ Kolmogorov and Wasserstein distance bounds for this distributional approximation. Stein's method has also been used to provide explicit error bounds in the approximation of geometric random sums by a geometric distribution \cite{dl25,prr_geometric}, in the setting of approximation by a geometric random sum \cite{d10,d16}, gamma approximation of negative binomial sums \cite{lx22}, and in various approximation results for random sums more generally \cite{d22,d15_clt,rollin05}.

When the sequence of random variables $X_1,X_2,\ldots$ are assumed to have zero mean and a non-zero and finite variance, the geometric random sum $\sqrt{p}\sum_{i=1}^{N_p} X_i$ converges in distribution to the Laplace distribution (see \cite[Proposition 2.2.9]{kkp01} and \cite[Theorem 2.1]{toda}). Following the approach of \cite{pekoz1} in the Laplace setting, \cite[Theorem 1.3]{pike} quantified this distributional approximation with an explicit $O(p^{1/2})$ bound in the bounded Wasserstein distance, whilst \cite{g21} and \cite{s21} used a refined application of Stein's method for Laplace approximation to obtain $O(p^{1/2})$ bounds in the stronger Kolmogorov and Wasserstein metrics, and the former work also obtained a $O(p)$ bound in the $d_2$ metric under a vanishing third moment assumption. Further bounds quantifying the Laplace approximation of the geometric random sum $\sqrt{p}\sum_{i=1}^{N_p} X_i$ are given in \cite{bu24,d15,g20,ks12,s18}, derived via alternative approaches.


More generally, a suitably standardised geometric random sum of i.i.d.\ random variables with small non-zero mean and a non-zero and finite variance converges in distribution to an asymmetric Laplace distribution. More specifically, suppose that $X_1,X_2,\ldots$ is an i.i.d.\ sequence of random variables with $\mathbb{E}[X_1]=0$ and $\mathrm{Var}(X_1)=\sigma^2\in(0,\infty)$. Then Proposition 3.4.4 of \cite{kkp01} states that, as $p\rightarrow0$,
\begin{align}\label{daaa}
\sqrt{p}\sum_{i=1}^{N_p}(X_i+\sqrt{p}a)\rightarrow_d \mathrm{AL}(0,a,\sigma),  
\end{align}
where we have made an obvious abuse of notation. More recently, \cite[Theorem 2.1]{toda} established this convergence in distribution result under weaker Lindeberg-type conditions. In the following theorem, we 
quantify the distributional approximation (\ref{daaa}) by providing $O(p^{1/2})$ Kolmogorov and Wasserstein distance bounds, and faster $O(p)$ convergence rates with respect to the $d_2$ metric under an additional vanishing third moment condition, for independent $X_i$. 
\begin{theorem}\label{thm4.1}
	Let $X_1,X_2,\ldots$ be a sequence of independent random variables with $\mathbb{E}[X_i]=0$, $\mathrm{Var}(X_i)=\sigma^2\in(0,\infty)$, for $i=1,2,\ldots$. Denote $\rho_k=\sup_{i\geq1}\mathbb{E}[|X_i|^{k}]$, for $k\geq1$.
    Suppose that $N\sim \mathrm{Geom}(p),$ $p\in(0,1),$ is independent of the  $X_i$'s. Set $W=\sqrt{p}\sum_{i=1}^{N}(X_i+\sqrt{p}a)$
    and let $Z\sim\mathrm{AL}(0,a,\sigma)$.

\vspace{3mm}

\noindent (i) Let $Y_i=X_i+\sqrt{p}a$, for $i=1,2\ldots$. Then, under the above assumptions, and the additional assumption that $\sigma^2>p a^2$, the following bound holds:
	\begin{align}
		d_\mathrm{K}(W,Z)&\leq\sqrt{p}\bigg(\frac{14}{\sqrt{2\sigma^2+a^2}}+\frac{7 |a|}{\sigma^2}\bigg)\sup_{i\geq1}\big\|F_{Y_i}^{-1}-F_{Y_i^A}^{-1}\big\|,\label{kb}
	\end{align}
  where $F^{-1}_Y$ denotes the generalised inverse of the cumulative distribution function of the random variable $Y$. 
  
  We now drop the assumption that $\sigma^2>p a^2$, but invoke the stronger moment condition $\rho_{k+2}<\infty$, for some $k\geq1$. Then 
	\begin{align}
		d_\mathrm{K}(W,Z)
         &\leq 4p^\frac{k}{2(k+1)}\bigg(\frac{14}{\sqrt{2\sigma^2+a^2}}+\frac{7 |a|}{\sigma^2}\bigg)^{\frac{k}{k+1}}\bigg\{(k+k^{-k})\bigg(5+\frac{7|a|}{\sqrt{2\sigma^2+a^2}}\bigg) \frac{\rho_{k+2}}{\sigma^2}\bigg\}^{\frac{1}{k+1}}.\label{4.14}
	\end{align}
	\noindent (ii) Suppose now that $\rho_3<\infty$. Then
	\begin{align}
		d_\mathrm{W}\big(W,Z\big)&\leq 2\sqrt{p}\bigg\{\sqrt{\sigma^2+a^2}+\frac{8(1+3\sqrt{p}|a|/\sigma)(\rho_3+|a|^3)}{3\sigma^2}\bigg\}.\label{wb}
	\end{align}
	\noindent (iii) Suppose now that $X_1,X_2,\ldots$ are i.i.d.\ with $\mathbb{E}[X_1^3]=0$ and $\mathbb{E}[X_1^4]<\infty$. Also, assume that $\sigma^2>pa^2$. Then
	\begin{align}
		d_2(W,Z)&\leq p\bigg\{\frac{\sigma^2}{\sqrt{2\sigma^2+a^2}}\bigg(\frac{\sqrt{2}\sigma}{1-p}+\frac{\sigma\sqrt{p}\log(1/p)}{1-p}\bigg(2+\frac{\mathbb{E}[|X_1|^3]}{\sigma^3}\bigg)\nonumber\\
		&\quad+|a|+\frac{2p|a|^3}{3(\sigma^2-pa^2)}\bigg)+\sigma^2+pa^2+\frac{\mathbb{E}[X_1^4]}{6(\sigma^2-pa^2)}\bigg\}.\label{4.13} 
	\end{align}
\end{theorem}

\begin{remark}  Inequalities (\ref{kb}) and (\ref{4.13}) hold under the assumption that $\sigma^2>pa^2$. The bound (\ref{kb}) is stated in terms of the generalised inverse of the cumulative distribution function of the random variables $Y_i^A$, $i\geq1$, and the assumption $\sigma^2>pa^2$ ensures that these random variables exist. Moreover, the assumption $\sigma^2>pa^2$ under which inequalities (\ref{kb}) and (\ref{4.13}) are stated is rather mild since $\sqrt{p}\sum_{i=1}^{N}(X_i+\sqrt{p}a)\rightarrow_d Z\sim\mathrm{AL}(0,a,\sigma)$ as $p\rightarrow0$, for fixed $a\in\mathbb{R}$ and $\sigma^2>0$.
\end{remark}

\begin{remark} The Kolmogorov distance bound (\ref{kb}) is of the optimal order $O(p^{1/2})$, but in applications one would be required to compute the quantity $\sup_{i\geq1}\|F_{Y_i}^{-1}-F_{Y_i^A}^{-1}\|$. We note that this quantity is easily bounded when the $Y_i$'s are bounded. Indeed, suppose that $|Y_i|\leq C$ for all $i\geq1$. Then, by part (iii) of Proposition \ref{prop3.3}, $|Y_i^A|\leq C$ for all $i\geq1$. We therefore have the cheap bound $\sup_{i\geq1}\|F_{Y_i}^{-1}-F_{Y_i^A}^{-1}\|\leq 2C$, from which we obtain the Kolmogorov distance bound
\begin{equation*}
d_\mathrm{K}(W,Z)\leq 14C\sqrt{p}\bigg(\frac{2}{\sqrt{2\sigma^2+a^2}}+\frac{|a|}{\sigma^2}\bigg).    
\end{equation*}

On the other hand, if $\|F_{Y_i}^{-1}-F_{Y_i^A}^{-1}\|$ is difficult to compute, one may prefer to use inequality (\ref{4.14}). The bound (\ref{4.14}) is of the worse order $O(p^{k/(2(k+1))})$. However, we observe that as $k$ increases (provided the moment condition $\rho_{k+2}<\infty$ is met) the rate of convergence improves, with the exponent $k/(2(k+1))$ approaching the optimal exponent $1/2$ as $k\rightarrow\infty$.
\end{remark}


\begin{proof} 
Throughout the proof we will use the representation for $W-W^A$ given by Lemma \ref{lem3.4}.

\vspace{3mm}

\noindent
(i) By inequality (\ref{3.51}) of Theorem \ref{thm3.5}, we have the bound 
	\begin{align}
		d_\mathrm{K}(W,Z)\leq \bigg(\frac{14}{\sqrt{2\sigma^2+a^2}}+\frac{7 |a|}{\sigma^2}\bigg)\beta+\bigg(5+\frac{7|a|}{\sqrt{2\sigma^2+a^2}}\bigg)\mathbb{P}(|W-W^A|>\beta).\label{kbnd}
	\end{align}
	Upon setting $\beta=\sqrt{p}\sup_{i\geq1}\|F_{Y_i}^{-1}-F^{-1}_{Y_i^A}\|$ in inequality (\ref{kbnd}) we deduce inequality (\ref{kb}) by an application of Strassen's theorem. 

    We now prove inequality (\ref{4.14}). We first prove the inequality under the assumption that $\sigma^2\geq 3pa^2$ (so that $\sigma^2>pa^2$, which is needed to ensure the existence of $Y_i^A$ for all $i\geq1$).  
    We have that, for $k\geq1$,
	\begin{align}
		\mathbb{E}\big[|W-W^A|^k\big]&\leq p^{k/2}\mathbb{E}\big[|Y_N-Y_N^A|^k\big]\leq 2^{k-1}p^{k/2}\big(\mathbb{E}[|Y_N|^k]+\mathbb{E}[|Y_N^A|^k]\big)\nonumber\\
		&\leq2^{k-1}p^{k/2}\bigg(\mathbb{E}[|Y_N|^k]+\frac{2(1+(k+2)\sqrt{p}|a|/\sigma)\mathbb{E}[|Y_N|^{k+2}]}{(k+1)(k+2)(\sigma^2-pa^2)}\bigg)\nonumber \\
        &\leq2^{k-1}p^{k/2}\bigg(\mathbb{E}[|Y_N|^k]+\frac{(1+\sqrt{3})\mathbb{E}[|Y_N|^{k+2}]}{2\sigma^2}\bigg),\label{junl}
	\end{align}
	where we used the basic inequality $|x+y|^{r}\leq2^{r-1}(|x|^r+|y|^r)$, $r\geq1$, and the absolute moment bound (\ref{simplebd}), and we obtained the simplification in the final step by using that $k\geq1$ and the assumption that $\sigma^2\geq 3pa^2$.
    We now obtain a simpler bound for $\mathbb{E}[|W-W^A|^k]$. By H\"older's inequality, we have the bound $\sigma^2\mathbb{E}[|Y_N|^k]\leq \mathbb{E}[|Y_N|^{k+2}]$. 
    Also, by using  H\"older's inequality and the assumption that $\sigma^2\geq 3pa^2$,
    \begin{align*}\mathbb{E}[|Y_N|^{k+2}]&\leq 2^{k+1}(\mathbb{E}[|X_N|^{k+2}]+p^{k/2+1}|a|^{k+2})\leq2^{k+1}(\rho_{k+2}+p^{k/2+1}|a|^{k+2})\\
    &\leq 2^{k+1}(\rho_{k+2}+(1/3)^{k/2+1}\sigma^{k+2})\leq 2^{k+1}(1+3^{-3/2})\rho_{k+2},
    \end{align*}
    where the final inequality holds since $k\geq1$.
Applying these bounds to inequality (\ref{junl}) gives that
\begin{align}
\mathbb{E}\big[|W-W^A|^k\big]&\leq2^{k-1}p^{k/2}\bigg(\frac{\mathbb{E}[|Y_N|^{k+2}]}{\sigma^2}+\frac{(1+\sqrt{3})\,\mathbb{E}[|Y_N|^{k+2}]}{2\sigma^2}\bigg)\nonumber\\
&\leq 2^{k-1}p^{k/2}\bigg(1+\frac{1+\sqrt{3}}{2}\bigg)\cdot2^{k+1}(1+3^{-3/2})\frac{\rho_{k+2}}{\sigma^2} \nonumber\\
&\leq 2^{2k+2}p^{k/2}\frac{\rho_{k+2}}{\sigma^2}. \label{jul}
\end{align}
Substituting the bound (\ref{jul}) into the bound (\ref{kolb}) of Theorem \ref{thm3.5}
yields inequality (\ref{4.14}) under the assumption $\sigma^2\geq 3pa^2$.

We now prove that inequality (\ref{4.14}) holds when $0<\sigma^2\leq 3pa^2$. Under this assumption,
\begin{align*}
&4p^\frac{k}{2(k+1)}\bigg(\frac{14}{\sqrt{2\sigma^2+a^2}}+\frac{7 |a|}{\sigma^2}\bigg)^{\frac{k}{k+1}}\bigg\{(k+k^{-k})\bigg(5+\frac{7|a|}{\sqrt{2\sigma^2+a^2}}\bigg) \frac{\rho_{k+2}}{\sigma^2}\bigg\}^{\frac{1}{k+1}}    \\
&\quad\geq4\bigg(\frac{\sigma}{\sqrt{3}|a|}\bigg)^{\frac{k}{k+1}}\bigg(\frac{7 |a|}{\sigma^2}\bigg)^{\frac{k}{k+1}}(10\sigma^k)^{\frac{1}{k+1}}>4\bigg(\frac{7}{\sqrt{3}}\bigg)^{\frac{k}{k+1}}\geq4\bigg(\frac{7}{\sqrt{3}}\bigg)^{1/2}>1\geq d_\mathrm{K}(W,Z),
\end{align*}
where in the first inequality we used that $\rho_{k+2}\geq\sigma^{k+2}$ by H\"older's inequality.

    \vspace{3mm}
	
	\noindent (ii) We now prove inequality (\ref{wb}), again considering the cases $\sigma^2\geq 2pa^2$ and $0<\sigma^2\leq 2pa^2$ separately. Firstly, we suppose that $\sigma^2\geq 2pa^2$. Then arguing similarly to part (i) of the proof we get that
    \begin{align}
     \mathbb{E}|Y_N-Y_N^A|&\leq\mathbb{E}|Y_N|+\mathbb{E}|Y_N^A|\leq\sqrt{\mathbb{E}[Y_N^2]}+\frac{(1+3\sqrt{p}|a|/\sigma)\mathbb{E}[|Y_N|^3]}{3(\sigma^2-pa^2)} \nonumber\\
     &\leq\sqrt{\sigma^2+pa^2}+\frac{4(1+3\sqrt{p}|a|/\sigma)(\rho_{3}+p^{3/2}|a|^3)}{3(\sigma^2-pa^2)}\nonumber\\
     &\leq 
     \sqrt{\sigma^2+a^2}+\frac{8(1+3\sqrt{p}|a|/\sigma)(\rho_{3}+|a|^3)}{3\sigma^2},\label{11f}
    \end{align}
where in the last step we used the assumption that $\sigma^2\geq 2pa^2$ and also that $p<1$. Substituting inequality (\ref{11f}) into inequality (\ref{3.53}) now yields the desired bound (\ref{wb}).

Suppose now that $0<\sigma^2\leq 2pa^2$. We have the cheap bound
\begin{align*}
d_\mathrm{W}(W,Z)\leq \mathbb{E}|W-Z|\leq \mathbb{E}|W|+\mathbb{E}|Z|\leq \sqrt{\mathbb{E}[W^2]}  +\sqrt{\mathbb{E}[Z^2]} \leq 2\sqrt{2a^2+\sigma^2},  
\end{align*}
where in obtaining the last inequality we used that $\mathbb{E}[W^2]=\mathrm{Var}(W)+(\mathbb{E}[W])^2=\sigma^2+(2-p)a^2\leq \sigma^2+2a^2$, and that $\mathbb{E}[W]=a$ and $\mathrm{Var}(W)=\sigma^2+(1-p)a^2$ (see the proof of Lemma \ref{lem3.4}). By H\"older's inequality we have that $\rho_3\geq\sigma^3$, and therefore, under the assumption that $0<\sigma^2\leq 2pa^2$,
\begin{align*}
&2\sqrt{p}\bigg\{\sqrt{\sigma^2+a^2}+\frac{8(1+3\sqrt{p}|a|/\sigma)(\rho_3+|a|^3)}{3\sigma^2}\bigg\}\\
&\quad\geq\frac{\sqrt{2}\sigma}{|a|}\bigg\{\sqrt{\sigma^2+a^2}+\frac{8(\sigma^3+|a|^3)}{3\sigma^2}\bigg\}=:C(a,\sigma^2).    
\end{align*}
By treating the cases $\sigma^2\geq a^2$ and $0<\sigma^2<a^2$ separately, it is easily seen that $C(a,\sigma^2)\geq 2\sqrt{2a^2+\sigma^2}\geq d_\mathrm{W}(W,Z)$, and we have therefore verified that inequality (\ref{wb}) holds for $0<\sigma^2\leq 2pa^2$.

 \vspace{3mm}
	
	\noindent (iii) We now suppose that $X_1,X_2,\ldots$ are i.i.d.\ with $\mathbb{E}[X_1^3]=0$ and $\mathbb{E}[X_1^4]<\infty$, so that $Y_1,Y_2,\ldots$ are also independent and identically distributed. We also assume that $\sigma^2>pa^2$, which ensures that $W^A$ and $Y_N^A$ exist. By inequality (\ref{3.56}) of Theorem \ref{thm3.5} we have the bound
	\begin{align}
		{d}_2(W,Z)\leq\frac{\sigma^2}{\sqrt{2\sigma^2+a^2}}\mathbb{E}\big[|\mathbb{E}[W-W^A\,|\,W]|\big]+\mathbb{E}\big[(W-W^A)^2\big].
			\end{align}
	We therefore require bounds for $\mathbb{E}[|\mathbb{E}[W-W^A\,|\,W]|]$ and $\mathbb{E}[(W-W^A)^2]$, and we start by bounding the latter expectation. By the independence of $Y_N$ and $Y_N^A$ we obtain that
	\begin{align}
		\mathbb{E}\big[(W-W^A)^2\big]=p\mathbb{E}\big[(Y_N^A-Y_N)^2\big]&=p\big(\mathbb{E}[Y_N^2]+\mathbb{E}[(Y_N^A)^2]\big)\nonumber\\
		&=p\bigg(\sigma^2+pa^2+\frac{\mathbb{E}[Y_1^4]-4\sqrt{p}a\mathbb{E}[Y_1^3]}{6(\sigma^2-pa^2)}\bigg), \label{pppp}
	\end{align}
	where in the final step we used the moment relation (\ref{mr}).
	Then, from the assumption that $\mathbb{E}[X_1^3]=0$, we obtain that 
    \begin{align*}
		\mathbb{E}[Y_1^3]=\mathbb{E}\big[(X_1+\sqrt{p}a)^3\big]&=3\mathbb{E}[X_1^2]\sqrt{p}a+p^{3/2}a^3=3\sigma^2\sqrt{p}a+p^{3/2}a^3,
	\end{align*}
    so that
    \begin{align}\label{qqqq}
     \mathbb{E}[Y_1^4]-4\sqrt{p}a\mathbb{E}[Y_1^3]=\mathbb{E}[X_1^4]+6pa^2\sigma^2+p^2a^4-4\sqrt{p}a(3\sigma^2\sqrt{p}a+p^{3/2}a^3)\leq\rho_4,
    \end{align}
    where in obtaining the equality we used that $\mathbb{E}[X_1]=\mathbb{E}[X_1^3]=0$ and $\mathbb{E}[X_1^2]=\sigma^2$. From (\ref{pppp}) and (\ref{qqqq}) we obtain the bound
	\begin{align}
		\mathbb{E}\big[(W-W^A)^2\big]=p\bigg(\sigma^2+pa^2+\frac{\rho_4}{6(\sigma^2-pa^2)}\bigg).\label{2eq}
	\end{align}
    
	We now obtain a bound for the conditional expectation $\mathbb{E}[|\mathbb{E}[W-W^A\,|\,W]|]$. Since $Y_N$ and $W^A$ are independent, we have that 
	\begin{align}\label{eqq1}
		\mathbb{E}[W-W^A\,|\,W]=\sqrt{p}\mathbb{E}[Y_N-Y_N^A\,|\,W]=\sqrt{p}\big(\mathbb{E}[Y_N\,|\,W]-\mathbb{E}[Y_N^A]\big).
	\end{align}
    By the moment relation (\ref{mr}) we obtain that 
	\begin{align}\label{eqq2}
		\mathbb{E}[Y_N^A]=\frac{\mathbb{E}[Y_N^3]-3\sqrt{p}a\mathbb{E}[Y_N^2]}{3(\sigma^2-pa^2)}=-\frac{2p^{3/2}a^3}{3(\sigma^2-pa^2)},
	\end{align}
	whilst applying  the tower property of conditional expectation gives that 
	\begin{align}\label{eqq3}
		\mathbb{E}[Y_N\,|\,W]=\mathbb{E}\big[\mathbb{E}[Y_N\,|\,W,N]\,|\,W\big]=\mathbb{E}\bigg[\frac{W}{\sqrt{p}N}\,\bigg|\,W\bigg],
	\end{align}
	where the last step follows since the $Y_i$'s are i.i.d.\ and hence exchangeable.
	With the equalities (\ref{eqq1}), (\ref{eqq2}) and (\ref{eqq3}) we obtain that
\begin{align}
 \mathbb{E}\big[|\mathbb{E}[W-W^A\,|\,W]|\big]
		&\leq\mathbb{E}\bigg[\mathbb{E}\bigg[\frac{|W|}{N}\,\bigg|\,W\bigg]\bigg]+\frac{2p^2|a|^3}{3(\sigma^2-pa^2)}=\mathbb{E}\bigg[\frac{|W|}{N}\bigg]+\frac{2p^2|a|^3}{3(\sigma^2-pa^2)} \nonumber  \\
        &=\sqrt{p}\sum_{n=1}^{\infty}\frac{1}{n}\mathbb{E}\bigg|\sum_{i=1}^{n}X_i+\sqrt{p}a\bigg|\mathbb{P}(N=n)+\frac{2p^2|a|^3}{3(\sigma^2-pa^2)}\nonumber\\
		&\leq\sqrt{p}\sum_{n=1}^{\infty}\frac{1}{n}\mathbb{E}\bigg|\sum_{i=1}^{n}X_i\bigg|\mathbb{P}(N=n)+p|a|\sum_{n=1}^{\infty}\mathbb{P}(N=n)+\frac{2p^2|a|^3}{3(\sigma^2-pa^2)}\nonumber\\
		&\leq\frac{\sqrt{2}\sigma p}{1-p}+\frac{\sigma p^{3/2}\log(1/p)}{1-p}\bigg(2+\frac{\mathbb{E}[|X_1|^3]}{\sigma^3}\bigg)+p|a|+\frac{2p^2|a|^3}{3(\sigma^2-pa^2)}, \label{de2}
\end{align}
where in the last step we used the following bound that was given in the proof of Theorem 1.1 in \cite{g21}:
\begin{equation*}
	\sqrt{p}\sum_{n=1}^{\infty}\frac{1}{n}\mathbb{E}\bigg|\sum_{i=1}^{n}X_i\bigg|\mathbb{P}(N=n)\leq\frac{\sqrt{2}\sigma p}{1-p}+\frac{\sigma p^{3/2}\log(1/p)}{1-p}\bigg(2+\frac{\mathbb{E}[|X_1|^3]}{\sigma^3}\bigg).
\end{equation*}	
	Finally, substituting inequalities (\ref{2eq}) and (\ref{de2}) into inequality (\ref{3.56}) yields the desired bound (\ref{4.13}).
\end{proof}

\subsection{Sums of independent random variables with random standardisation}\label{sec4.2}


In Section \ref{sec4.1}, we considered the asymmetric Laplace approximation of a geometric random sum with deterministic standardisation. Similarly, we may arrive at an asymmetric Laplace limit if we apply a suitable random standardisation to a deterministic sum of independent random variables. More specifically, let $X_1,\ldots,X_n$ be independent random variables with zero mean and variance $\sigma^2\in(0,\infty)$. For $n\geq2$, let $B_{n-1}\sim \mathrm{Beta}(1,n-1)$ (with density $f_{B_n}(x)=n(1-x)^{n-1}$, $0<x<1$) be independent of the $X_i$'s.  Then 
\begin{equation*}\label{ggg}
	\sqrt{B_{n-1}}\bigg(\sum_{i=1}^{n}X_i+a\sqrt{B_{n-1}}\bigg)\rightarrow_d\mathrm{AL}(0,a,\sigma),\quad n\rightarrow\infty.
\end{equation*}
This convergence in distribution result is proved in the case $a=0$ in Proposition 2.2.12 of \cite{kkp01}. Moreover, $O(n^{-1/2})$ Kolmogorov and Wasserstein distance bounds and $O(n^{-1})$ bounds in the $d_{1,2}$ metric under a vanishing third moment assumption in this $a=0$ case were given in Theorem 1.3 of \cite{g21}. In the following theorem, we apply these bounds to derive explicit bounds to quantify the convergence in distribution result (\ref{ggg}) for general $a\in\mathbb{R}$.

\begin{theorem}\label{thm4.3} For $n\geq 2$, let $X_1,\ldots,X_n$ be independent random variables such that $\mathbb{E}[X_i]=0$, $\mathrm{Var}(X_i)=\sigma^2\in(0,\infty)$ and $\mathbb{E}[|X_i|^3]<\infty$, for $i=1,\ldots,n$. For $a\in\mathbb{R}$, let \[W=\sqrt{B_{n-1}}\left(\sum_{i=1}^{n}X_i+a\sqrt{B_{n-1}}\right),\] where $B_{n-1}\sim\mathrm{Beta}(1,n-1)$, and also let $Z\sim\mathrm{AL}(0,a,\sigma)$.

\vspace{3mm}

\noindent (i) Under the above assumptions,
	\begin{align}
		d_\mathrm{K}(W,Z)&\leq \frac{0.56}{\sigma^3n^{3/2}}\sum_{i=1}^{n}\mathbb{E}[|X_i|^3]+\frac{5}{n}.
        \label{5.11}
        \end{align}
\noindent (ii) Under the above assumptions,
        \begin{align}
		d_\mathrm{W}(W,Z)&\leq \frac{2\sqrt{2}}{3\sigma^2 n^{3/2}}\sum_{i=1}^{n}\mathbb{E}[|X_i|^3]+\frac{9.168\sigma}{n}+\frac{2|a|}{n+1}.\label{5.12}
	\end{align}
\noindent (iii) Suppose further that $\mathbb{E}[X_i^3]=0$ and $\mathbb{E}[X_i^4]<\infty$, for $1\leq i\leq n$. Then
	\begin{align}
		d_{1,2}(W,Z)\leq\frac{\sigma^2}{3n^2}\sum_{i=1}^{n}\bigg(3+\frac{\mathbb{E}[X_i^4]}{\sigma^4}\bigg)+\frac{9.168\sigma}{n}+\frac{2|a|}{n+1}.\label{5.13}
	\end{align}
\end{theorem}
\begin{proof} (i)
	Let $U_n=\sqrt{nB_{n-1}}$ and $V_n=n^{-1/2}\sum_{i=1}^{n}X_i$. Also, let $U$ be a Rayleigh random variable with density $f_U(u)=2u\mathrm{e}^{-u^2}$, $u>0$, and let $V\sim \mathrm{N}(0,\sigma^2)$. Then $W=_dU_n(V_n+aU_n)=U_nV_n+aU_n^2$, and $Z=_dUV+aU^2$ by the normal variance-mean mixture representation (\ref{nvm}) of the asymmetric Laplace distribution. 
	We write $W_1=U_nV_n$, $Z_1=UV$, $W_2=aU_n^2$ and $Z_2=aU^2$, so that 
	\begin{align}
		d_\mathrm{K}(W,Z)=d_\mathrm{K}(W_1+W_2,Z_1+Z_2)\leq d_\mathrm{K}(W_1,Z_1)+d_\mathrm{K}(W_2,Z_2).\label{kb0}
	\end{align}
	We now bound $d_\mathrm{K}(W_2, Z_2)$. By the scale invariance of the Kolmogorov distance we have that $d_\mathrm{K}(W_2, Z_2)=d_\mathrm{K}(aU_n^2,aU^2)=d_\mathrm{K}(U_n^2,U^2)=d_\mathrm{K}(B,E)$, where $B=_d nB_{n-1}$ and $E\sim \mathrm{Exp}(1)$ and we used the standard distributional relation that $U^2=_d E$. 
    To bound $d_\mathrm{K}(B,E)$ we will use the comparison of Stein operators approach with the following Stein operators for the random variables $B$ and $E$:
	\begin{align}
		\mathcal{A}_{B}f(x)&=x(1-x/n)f'(x)+(1-x)f(x),\label{bop}\\
		\mathcal{A}_{E}f(x)&=xf'(x)+(1-x)f(x).\label{eop}
	\end{align}
	The Stein operator (\ref{bop}) for $B$ is obtained by applying a straightforward rescaling to the Stein operator of \cite{d15,gr13} for the beta distribution, and the Stein operator (\ref{eop}) for the $\mathrm{Exp}(1)$ distribution can be found in \cite{chatterjee}. 
	Let $h\in\mathcal{H}_\mathrm{K}$ and let $f_h$ denote the solution of the $\mathrm{Exp}(1)$ Stein equation $\mathcal{A}_Ef(x)=h(x)-\mathbb{E}[h(E)]$. We then have that 
	\begin{align}
		d_\mathrm{K}(W_2, Z_2)=d_\mathrm{K}(B,E)&=\sup_{h\in\mathcal{H}_\mathrm{K}}|\mathbb{E}[\mathcal{A}_Ef_h(B)]|\nonumber\\
		&=\sup_{h\in\mathcal{H}_\mathrm{K}}|\mathbb{E}[\mathcal{A}_Ef_h(B)]-\mathbb{E}[\mathcal{A}_Bf_h(B)]|\nonumber\\
		&=\sup_{h\in\mathcal{H}_\mathrm{K}}|\mathbb{E}[(B^2/n)f_h'(B)]|\nonumber\\
		&\leq\frac{1}{n}\|Bf_h'(B)\|\mathbb{E}[B]\leq\frac{2}{n}\|h\|\mathbb{E}[B]=\frac{2}{n},\label{kb2}
	\end{align}
	where in the penultimate step we used the bound $\|xf_h'(x)\|\leq2\|h-\mathbb{E}[h(E)]\|\leq 2$ (see inequality (2.5) of \cite{gaunt chi square}). Combining the bounds (\ref{kb0}) and (\ref{kb2}) and a simplification of the bound given in Theorem 1.3 of \cite{g21} for $d_\mathrm{K}(W_1,Z_1)$ (the simplification follows from employing the basic bound $(1-2/n)^{n-2}\leq1$, for $n\geq2$) yields inequality (\ref{5.11}). 
    
    \vspace{3mm} 
    
    \noindent (ii) We proceed similarly to part (i) of the proof. We note that the analogue of inequality (\ref{kb0}) is given by
	\begin{align}
		d_\mathrm{W}(W,Z)\leq d_\mathrm{W}(W_1,Z_1)+d_\mathrm{W}(W_2,Z_2).\label{Wte}
    \end{align}
    For the term $d_\mathrm{W}(W_2,Z_2)$, we first note that $d_\mathrm{W}(W_2,Z_2)=d_\mathrm{W}(aU_n^2,aU^2)=|a|d_\mathrm{W}(U_n^2,U^2)=|a|d_\mathrm{W}(B,E)$.
Proceeding similarly to in part (i) but instead taking $h\in\mathcal{H}_\mathrm{W}$ we have that
	\begin{align}
		d_\mathrm{W}(W_2,Z_2)=|a|d_\mathrm{W}(B,E)&
		\leq |a|\sup_{h\in\mathcal{H}_\mathrm{W}}\mathbb{E}[|(B^2/n)f_h'(B)|]\nonumber\\
		&\leq\frac{|a|}{n}\|f_h'\|\mathbb{E}[B^2]\leq\frac{|a|}{n}\mathbb{E}[B^2]=\frac{2|a|}{n+1},\label{Wb2}
	\end{align}
	where in the penultimate step we used the bound $\|f_h'\|\leq\|h'\|\leq1$ (see Remark 3.18 of \cite{d15}). By combining inequality (\ref{Wte}) with inequality (\ref{Wb2}) and an improvement of the bound given in Theorem 1.3 of \cite{g21} for $d_\mathrm{W}(W_1,Z_1)$ we obtain inequality (\ref{5.12}). The improved bound for $d_\mathrm{W}(W_1,Z_1)$ is obtained by upgrading inequality (4.5) of \cite{g21}, which is used in the derivation of the bound for $d_\mathrm{W}(W_1,Z_1)$, by applying the bound of Corollary 4.2 of \cite{chen} rather than the bound of Theorem 2.1 of \cite{rcoupling}.

\vspace{3mm} 
    
    \noindent (iii) Since
$d_{1,2}(W_2,Z_2)\leq d_\mathrm{W}(W_2,Z_2)$ we have the bound
	\begin{align}
		d_{1,2}(W,Z)\leq d_{1,2}(W_1,Z_1)+d_\mathrm{W}(W_2,Z_2).\label{12te}
	\end{align}
By combining (\ref{12te}) with (\ref{Wb2}) and the bound given in Theorem 1.3 of \cite{g21} for the term $d_{1,2}(W_1,Z_1)$ we obtain the desired bound (\ref{5.13}).
\end{proof}

\subsection*{Acknowledgements}
RG is funded in part by EPSRC grant EP/Y008650/1 and EPSRC grant UKRI068. HS is funded by EPSRC grant EP/Y008650/1.

\end{document}